\documentclass{elsarticle}

\usepackage{lineno,hyperref}
\modulolinenumbers[5]

\journal{Computers and Mathematics with Applications}

 \usepackage{graphicx}
 \usepackage{amssymb}
 \usepackage{amsmath}
 \usepackage{amsfonts}
 \usepackage{bbold}
 \usepackage{bbm}
 \usepackage{amsthm}
 \usepackage{color}
 \usepackage{soul}
 \usepackage{xcolor}

\newcommand{\change}[2]{#2} 

\def\c {{\boldsymbol c}}
\def\u {{\boldsymbol u}}

\def\ln{\textup{ln}}
\def\tr{\textup{tr}}

\def\one {\mathbbm{1}}
\def\R   {\mathbbm{R}}
\def\N   {\mathbbm{N}}

\def\sig {{\boldsymbol \sigma}}

\def \hueco{\noalign{\medskip}}
\def \beq{\begin{equation}}
\def \eeq{\end{equation}}
\def \ba{\begin{array}}
\def \ea{\end{array}}
\def \bal{\begin{aligned}}
\def \eal{\end{aligned}}
\def \bm{\begin{multline}}
\def \em{\end{multline}}
\def \dis{\displaystyle}

\def\softd{{\leavevmode\setbox1=\hbox{d}%
          \hbox to 1.05\wd1{d\kern-0.4ex{\char039}\hss}}}

\newtheorem{theorem}{Theorem}[section]
\newtheorem{lemma}{Lemma}[section]

\newtheorem{remark}{Remark}[section]
\newtheorem{definition}{Definition}[section]
          
\begin{document}

\begin{frontmatter}
 
\title{Energy-stable linear schemes for polymer-solvent phase field models}

\author[mymainaddress]{Paul J. Strasser\corref{mycorrespondingauthor}}
\cortext[mycorrespondingauthor]{Corresponding author}
\ead{strasser@uni-mainz.de}

\author[mysecondaryaddress]{Giordano Tierra}

\author[mytarziaryaddress]{Burkhard D\"unweg}

\author[mymainaddress]{M\'aria Luk\'a\v{c}ov\'a-Medvi\softd ov\'a}

\address[mymainaddress]{Institute of Mathematics, Johannes Gutenberg University Mainz, Staudingerweg 9, 55128 Mainz, Germany}

\address[mysecondaryaddress]{Department of Mathematics, Temple University, 1805 N. Broad Street, Philadelphia PA 19122, USA}

\address[mytarziaryaddress]{Max Planck Institute for Polymer Research, Ackermannweg 10, 55128 Mainz, Germany}

\begin{abstract}
We present new linear energy-stable numerical
  schemes for numerical simulation of complex polymer-solvent
  mixtures. The mathematical model proposed by Zhou, Zhang and E~{\sl
    (Physical Review E 73, 2006)} consists of the Cahn-Hilliard
  equation which describes dynamics of the interface that separates
  polymer and solvent and the Oldroyd-B equations for the
  hydrodynamics of polymeric mixtures. The model is thermodynamically
  consistent and dissipates free energy. Our main goal in this paper
  is to derive numerical schemes for the polymer-solvent mixture model
  that are energy dissipative and efficient in time. To this end we
  will propose several problem-suited time discretizations yielding
  linear schemes and discuss their properties.
\end{abstract}

\begin{keyword}
Two-phase flows \sep Non-Newtonian \sep Navier-Stokes \sep Cahn-Hilliard \sep 
	  Oldroyd-B \sep Flory-Huggins \sep Free energy dissipation \sep Linear schemes
\end{keyword}

\end{frontmatter}

\linenumbers

\section{Introduction}
\label{sec:intro}

Phase separation in binary fluids is a fundamental process in
condensed-matter physics. For Newtonian fluids the phenomenon of
spinodal decomposition is reasonably well understood in terms of the
so-called ``model H''~\cite{hohenberg,bray,onuki}, where the
hydrodynamic equations of motion for mass and momentum conservation
are coupled to a convection-diffusion equation for the concentration
(or in general the ``phase field'' variable $\phi$), and the
thermodynamics, which is described by a (free) energy functional
$E(\phi)$, gives rise to a driving force, see, e.~g.,
\cite{abels,GT,GT1,zhou}.  In such ``diffuse interface'' or ``phase
field'' models, the interface between two phases is a thin layer of
finite thickness, across which $\phi$ varies continuously.

A big advantage of such models is that interfaces are defined
implicitly and do not need to be tracked. Similarly, topological
changes of the interface structure are automatically described
correctly. The physics (and therefore also the mathematics and
numerics) becomes more involved if one component --- or both --- is a
macromolecular compound. In this case, the large molecular relaxation
time gives rise to a dynamic coupling between intra-molecular
processes and the unmixing on experimentally relevant time scales,
with interesting new phenomena, for which the term ``viscoelastic
phase separation''~\cite{tanaka} has been coined. Here the
construction of physically sound dynamic equations with suitable
constitutive relations to describe the viscoelasticity is already a
challenge in itself. \textit{Tanaka}~\cite{tanaka} made the first
attempt in this direction; however, \textit{Zhou et al.}~\cite{zhou}
showed later that this dynamics violates the second law of
thermodynamics and provided a corrected set of equations that satisfy
it. We thus study the diffuse-interface viscoelastic equations put
forward in~\cite{zhou} for the case of the unmixing process of a
polymer-solvent system.

Typically the interfacial region separating the two fluids is very
narrow, and a high spatial resolution is required to accurately
capture the interface dynamics. In fact, the underlying problem is
stiff, which necessitates an implicit time discretization. Moreover,
the solution admits several time scales over which it evolves,
cf.~\cite{KW}. In the literature one can find already several
numerical methods that have been used for the numerical approximation
of diffuse interface models, see, e.~g.,~\cite{CKQT,GT1,KW,lee,GT} and
the references therein.

In order to describe the dynamics of a complex polymer-solvent
mixture, the Cahn-Hilliard equations for the phase field evolution are
coupled with the Oldroyd-B equations, which consist of the momentum
equation for the velocity field, the continuity equation, and the
rheological equation for time evolution of the elastic stress
tensor. We note in passing that there is quite a large number of
analytical as well as numerical results available in the literature
for the Oldroyd-B system, see, e.~g.,~\cite{BB,FK,FGO,bangwei}. The
main challenge in this field is to obtain a stable approximate
numerical solution for large Weissenberg numbers. The dimensionless
Weissenberg number represents elastic effects; it is large when the
molecular relaxation time is comparable to the time scale of the flow,
or even exceeds it significantly. In the present work we consider the
non-critical regime of Weissenberg numbers. Applying the techniques
from~\cite{FK,bangwei}, a further generalization using the
log-transformation of the elastic stress tensor and the Lagrange-type
approximation of the convective term is possible.

The purpose of the present paper is to derive energy-stable and
runtime-efficient numerical schemes to solve the above-mentioned
equations. This task has already been tackled by us in a preliminary
fashion before~\cite{blurbfeb2017}, from which paper we have also
taken most of the wording of the present introduction. Compared to
Reference~\cite{blurbfeb2017}, we have significantly improved the
results, and also provide a much broader context and far more details.

The paper is organized in the following way. In Section~\ref{sec:math}
we present a mathematical model for the polymer-solvent mixture
consisting of the Cahn-Hilliard equation for the interface dynamics
and the Oldroyd-B equations for the hydrodynamics. We also introduce a
simplified model modelling only interface dynamics of the
polymer-solvent mixture without any hydrodynamic
effects. Section~\ref{sec:num} is devoted to problem-suited numerical
methods for both models. We present first and second order schemes
that are linear and energy dissipative.  
\change{Starting with the simplified model 
and continuing with the full model for polymer-solvent mixture.}
{We start with numerical methods for the simplified model and continue with corresponding methods for the full model for polymer-solvent mixture.}
For the latter we propose two types of linear, free energy
dissipative schemes, fully coupled schemes in
Subsection~\ref{sec:coupled} and the splitting scheme in
Subsection~\ref{sec:splitting}. Numerical experiments presented in
Section~\ref{sec:simulations} confirm the schemes robustness and
reliability to simulate viscoelastic phase separation.

\section{Mathematical models}\label{sec:math}

A classical approach to model interface problems is the diffuse
interface theory that describes the dynamics of the interfaces by
layers of small thickness whose structure is determined by a balance
of molecular forces. Here the tendencies for mixing and de-mixing are
in competition through a non-local mixing energy. Diffuse interface
models are able to treat topological changes of the interface in a
natural way. The surface motion is governed by the Cahn-Hilliard
equation, see \cite{ch}, that can be derived as the gradient flow of 
a phase-field free energy functional 
\beq E_{mix}(\phi) = \int_\Omega
\Big(\frac{C_0}2|\nabla\phi|^2 + F(\phi)\Big) \,.
\label{freeEnergy}
\eeq 
Here $\phi$ denotes the phase-field variable that is used to
express the two phases of the system. The phase-field function varies
smoothly over the interfacial regions. Further, $\Omega$ is a
computational domain with Lipschitz continuous boundary, $C_0$ is a
positive constant controlling the interface width and $F(\phi)$
denotes a double-well potential that represents the tendency of a
system to have two different stable phases.

A simple potential that satisfies these conditions is the
Ginzburg-Landau potential
\beq
 \label{GL_pot}
 F_{pol}(\phi) = \frac14(\phi^2-1)^2 \, ,
\eeq
which is defined on the whole real axis, and whose minima occur at
$\phi = \pm 1$. This potential is quite often studied in the
mathematical literature, see, e.~g., \textit{Elliot and
  Zheng}~\cite{elliot} or \textit{Elliot and
  Garcke}~\cite{eg1,eg2}. From a physical point of view, the
Flory-Huggins potential~\cite{pflory,huggins}
\beq\label{eq:fh} 
F_{log}(\phi) = \frac{1}{n_p}\phi \ln\phi +
\frac{1}{n_s}(1-\phi)\ln(1-\phi) + \chi \phi (1-\phi) \, ,
\eeq
which is defined on the interval $(0,1)$ and has two minima within,
describes polymer-solvent phase separation more
accurately, as it has been derived as a Mean Field theory for polymer
systems. In \change{Equation~}{}\eqref{eq:fh}, $n_p$ and $n_s$ denote the degrees
of polymerization of the two components, while $\chi > 0$ is the
(temperature-dependent) Flory-Huggins interaction parameter and we
measure the free energy in units of the thermal energy $k_B T$, where
$k_B$ is the Boltzmann constant and $T$ the absolute temperature.

For purposes of the proofs to be presented below, we mostly consider
the Ginzburg-Landau potential as defined in \change{Equation~}{}\eqref{GL_pot},
however with the modification proposed in, e.~g., \cite{convex, nematic},
where the steep increase 
$\sim \phi^4$ outside $[-1,1]$ is replaced by a weaker quadratic rise,
\begin{equation}\label{eq:GL_mod}
 \tilde{F}_{pol}=
 \begin{cases} 
 (\phi+1)^2 \quad \quad \phi < -1\,, \\
 \frac14(\phi^2-1)^2 \quad \phi \in [-1,1]\,, \\
 (\phi-1)^2 \quad \quad \phi > 1 \,.
 \end{cases}
\end{equation}
This modified potential is defined on the whole real axis and has a bounded 
second derivative. These properties facilitate to establish some bounds
needed in the proofs. Thus the Ginzburg-Landau potential allows the
derivation of schemes that are energy-stable even though they are
linear.  It is also possible to derive schemes that are more generally
applicable; one of these latter schemes is applied in our numerical
experiments, in which we use the Flory-Huggins potential in order to
facilitate comparisons with computer simulations of a quasi-atomistic
model~\cite{blurbfeb2017}.

Now, the Cahn-Hilliard equation can be derived from the
mass balance law
$$
  \frac{\partial \phi}{\partial t} = - \nabla \cdot J\,,
$$
where the mass flux $J$ is defined as
$$
  J = - m(\phi) \nabla \mu\,.
$$
Here
$$
  m(\phi) = M\left(\phi (1-\phi)\right)^n
$$
denotes the mobility function with $M$ a positive constant, $n\in\N_0$
and $\mu$ denotes the chemical potential such that
$$
\mu:=\frac{\delta E_{mix}}{\delta \phi} = -C_0 \Delta \phi + f(\phi)\,.
$$
Here $\frac{\delta E_{mix}}{\delta \phi}$ is the variational derivative of the mixing energy and $f(\phi) = F'(\phi)$. \\
Gathering this equations yields the Cahn-Hilliard equation
\beq \label{CH_eq}
\frac{\partial \phi}{\partial t} = \nabla \cdot \Big\{ M\left(\phi (1-\phi)\right)^n \, \nabla \Big[-C_0\Delta\phi +f(\phi)\Big] \Big\}\,.
\eeq	
Dynamics of Newtonian two-phase mixtures is usually described as the
gradient flow of the free energy consisting of the mixing energy
$E_{mix}$ and the kinetic energy $E_{kin}$. This leads to \change{a}{the} coupled
Cahn-Hilliard-Navier-Stokes system. In order to include the influence
of polymers in such a system we extend it to a viscoelastic phase
field model.  This has been done at first by
\textit{Tanaka}~\cite{tanaka} by adding viscoelastic energy due to the
bulk and shear stress, here the separation of the total stress tensor
into a bulk and a shear part was motivated by \textit{Tanaka and
  Araki}~\cite{Araki}. This model violates the second law of
thermodynamics, i.~e. it is not free energy dissipative. In the recent
paper \cite{zhou} \textit{Zhou, Zhang and E} propose an improved model
for the viscoelastic phase separation that is thermodynamically
consistent.

The total free energy $E$ is given as 
\begin{multline}
E_{tot}(\phi,q,\sig,\u)
= E_{mix}(\phi) + E_{kin}(\u) + E_{conf}(q) + E_{el}(\sig)
\\ \hueco
= \int_\Omega \Big(\frac{C_0}2|\nabla\phi|^2 + F(\phi)\Big) + \int_\Omega \frac12|\u|^2  + \int_\Omega
\frac12|q|^2 + \int_\Omega \frac12 \tr(\sig)
\,,
\end{multline}
where $\u$ is the averaged velocity field of the two components, 
$q$ the scalar bulk stress with $E_{conf}$ the corresponding chain 
conformational entropy of the polymer molecules and $\sig$ the shear 
stress tensor with $E_{el}$ the corresponding elastic energy of the polymer molecules.
Recalling that the chemical potential  $\mu = -C_0 \Delta \phi + f(\phi)$ 
and that we work with the Flory-Huggins potential \eqref{eq:fh}, 
i.e. $\phi \in (0,1)$, we obtain by the variational principle of the 
free energy minimization following the standard procedures of 
nonequilibrium thermodynamics, see \cite{zhou},
\beq
\bal\label{eq:model}
\frac{\partial\phi}{\partial t} + \u\cdot\nabla\phi &=
\nabla\cdot\left\{\phi(1-\phi)\,M\Big[\phi(1-\phi)\nabla\mu
- \nabla (A_1(\phi)\,q)\Big]\right\}\,,
\\ \hueco
\frac{\partial q}{\partial t} + \u\cdot\nabla q &=
- \frac1{\tau_b(\phi)} q - A_1(\phi)\,\nabla \cdot
  \left\{M\Big[\phi(1-\phi)\nabla\mu
- \nabla (A_1(\phi)\,q)\Big]\right\}\,,
\\ \hueco
\frac{\partial\sig}{\partial t} + (\u\cdot\nabla)\sig &=
(\nabla\u)\cdot\sig+\sig\cdot(\nabla\u)^T
- \frac1{\tau_s(\phi) }\sig
+ B_2(\phi)\Big[\nabla\u+(\nabla\u)^T\Big]\,,
\\ \hueco
\frac{\partial\u}{\partial t} + (\u\cdot\nabla)\u &=
- \nabla p + \nabla\cdot\left\{\eta(\phi)\Big[\nabla\u +
(\nabla\u)^T\Big]\right\} - \nabla\cdot(C_0\nabla\phi\otimes\nabla\phi)
+ \nabla\cdot\sig\,,
\\ \hueco
\nabla\cdot\u &= 0\,,
\eal
\eeq
where $\tau_b(\phi)=\tau_b^0\,\phi^2$ and $\tau_{s}(\phi)=\tau_{s}^0\,\phi^2$ are the relaxation times,
$B_2(\phi)=m^0_s\,\phi^2$ is the relaxation modulus, and $\tau_b^0, \tau_s^0$ and $m^0_s$ are positive constants.
$A_1(\phi)$ is the bulk modulus. The precise definition will be  given in Section~\ref{sec:simulations}.
Further,  $\eta(\phi)=1-\tau_s(\phi)B_2(\phi)$ is the viscosity
which is dependent on the relaxation and $p$ is the pressure.
For the aforementioned mobility the quartic function $m(\phi) = M\left(\phi(1-\phi)\right)^2$ is used. \\
\textit{Zhou et al.}~\cite{zhou} also considered the special case of model \eqref{eq:model} without hydrodynamic transport,
i.~e. $\u=0$. Note, that we use the same symbol $0$ for a scalar, a vector or a matrix. The resulting simplified model reads
\beq\label{eq:simp_model}
\ba{c}\dis
\frac{\partial\phi}{\partial t} =
\nabla\cdot\left\{\phi(1-\phi)M\Big[\phi(1-\phi)\nabla\mu - \nabla (A_1(\phi)\,
q)\Big]\right\}\,,
\\ \hueco\dis
\frac{\partial q}{\partial t}  = -\frac1{\tau_b(\phi)} q -
A_1(\phi)\nabla\cdot\left\{M\Big[\phi(1-\phi)\nabla\mu - \nabla (A_1(\phi)\,
q)\Big]\, \right. .
\ea
\eeq
In literature we can find already several well-established numerical methods for the Cahn-Hilliard equation (\ref{CH_eq}), 
see, e.~g.,~\cite{CKQT,GT1,KW,lee,GT}. 
\change{In the following Subsection we start by discussing the simplified model\eqref{eq:simp_model}.}
{In order to understand the crucial properties of the viscoelastic two-phase model \eqref{eq:model} we start by discussing its simplified version \eqref{eq:simp_model} in the following subsection.  
For the sake of simplicity we will call model \eqref{eq:model} the full model and its simplification \eqref{eq:simp_model} the simplified model.}

\subsection{\change{Simplified model without hydrodynamics}{Simplified model (without hydrodynamics)}}

In a special case when the hydrodynamics effects are neglected, i.~e. $\u = 0$, the total energy of the system consists of
the mixing energy and the chain conformational energy
\beq
E_{tot}(\phi,q)
=E_{mix}(\phi) + E_{conf}(q)
=\int_\Omega \Big(\frac{C_0}2|\nabla\phi|^2 + F(\phi)\Big) + \int_\Omega
\frac12|q|^2\,.
\eeq
The minimization principle yields the \change{}{simplified} model \eqref{eq:simp_model}.

\begin{theorem}\label{th:energylaw}
The problem \eqref{eq:simp_model} satisfies the following energy law
\beq\label{eq:energylaw}
\frac{dE_{tot}(\phi,q)}{dt} =
- \frac1{\tau^0_b} \left\|\frac{q}{\phi}\right\|^2_{L^2(\Omega)}
-\int_\Omega M\Big[\phi(1-\phi)\nabla\mu - \nabla (A_1(\phi)\, q)\Big]^2\,.
\eeq
\end{theorem}

\begin{proof}  
Multiplying \eqref{eq:simp_model}$_1$ by $\mu$ and integrating over the
  computational domain $\Omega$, assuming suitable boundary
  conditions (e.~g. periodic boundary conditions), and applying
  integration by parts we obtain
$$
\ba{cl}\dis
&\dis\int_\Omega \frac{\partial\phi}{\partial t}\mu
- \int_\Omega \nabla \cdot \left\{\phi(1-\phi)M\Big[\phi(1-\phi)\nabla\mu - \nabla (A_1(\phi)
\, q)\Big]\right\}\mu
\\ \hueco\dis
=&\dis\int_\Omega \frac{\partial\phi}{\partial t}\frac{\delta E_{mix}(\phi)}{\delta
\phi}
+ \int_\Omega \left\{\phi(1-\phi)M\Big[\phi(1-\phi)\nabla\mu - \nabla (A_1(\phi)
\, q)\Big]\right\}\nabla\mu
\\ \hueco\dis
=&\dis\frac{d E_{mix}(\phi)}{dt}
+ \int_\Omega M \Big[\phi(1-\phi)\nabla\mu - \nabla (A_1(\phi)\, q) \Big]
\Big[\phi(1-\phi)\nabla\mu\Big]
=0\,.
\ea
$$
Further, multiplying \eqref{eq:simp_model}$_2$ by $q$ and integrating over $\Omega$ with suitable
boundary conditions yields
$$\ba{cl}\dis
&\dis\int_\Omega\frac{\partial q}{\partial t}q  + \int_\Omega\frac1{\tau_b(\phi)} q^2
+ \int_\Omega A_1(\phi) \nabla\cdot\left\{M\Big[\phi(1-\phi)\nabla\mu - \nabla (A_1(\phi)
\, q)\Big]\right\} q
\\ \hueco\dis
=&\dis\int_\Omega\frac12\frac{\partial q^2}{\partial t}
+ \int_\Omega\frac1{\tau^0_b\phi^2} q^2
+ \int_\Omega \nabla \cdot M\Big[\phi(1-\phi)\nabla\mu - \nabla (A_1(\phi)\, q)\Big]
A_1(\phi)\,q
\\ \hueco\dis
=&\dis\frac{d}{dt}\left(\int_\Omega\frac12 q^2 \right)
+ \frac1{\tau^0_b}\int_\Omega \left(\frac{q}{\phi}\right)^2
\\ \hueco\dis
&\dis - \int_\Omega M\Big[\phi(1-\phi)\nabla\mu - \nabla (A_1(\phi)\, q)\Big]
\nabla(A_1(\phi)\,q)
\\ \hueco\dis
=&\dis\frac{d E_{conf}(q)}{dt}
+ \frac1{\tau^0_b} \left\|\frac{q}{\phi}\right\|^2_{L^2(\Omega)}
\\ \hueco\dis 
&\dis + \int_\Omega M\Big[\phi(1-\phi)\nabla\mu - \nabla (A_1(\phi)\, q)\Big] \Big[ -
\nabla(A_1(\phi)\,q)\Big]
=0\,.
\ea
$$
Then, adding both relations we obtain
\begin{multline*}
\frac{d E_{mix}(\phi)}{dt}
+ \frac{d E_{conf}(q)}{dt}
+ \frac1{\tau^0_b} \left\|\frac{q}{\phi}\right\|^2_{L^2(\Omega)}
\\ \hueco
+ \int_\Omega M\Big[\phi(1-\phi)\nabla\mu - \nabla (A_1(\phi)\, q)\Big]
\Big[\phi(1-\phi)\nabla\mu - \nabla(A_1(\phi)\,q) \Big]=0\,,
\end{multline*}
which is the desired energy law \eqref{eq:energylaw}.
\end{proof}

\subsection{\change{Two-phase model for viscoelastic flow}{Full model (with hydrodynamics)}}\label{sec:viscoelastic}

In this Subsection we consider the full two-phase model for viscoelastic phase separation \eqref{eq:model}.
The total free energy of the system consists of the mixing energy, the conformation energy, the elastic energy and the kinetic energy
\begin{multline}
E_{tot}(\phi,q,\sig,\u)
= E_{mix}(\phi) + E_{conf}(q) + E_{el}(\sig) + E_{kin}(\u)
\\  \hueco
= \int_\Omega \Big(\frac{C_0}2|\nabla\phi|^2 + F(\phi)\Big) + \int_\Omega
\frac12|q|^2 + \int_\Omega \frac12 \tr(\sig) + \int_\Omega \frac12|\u|^2
\,.
\end{multline}
In order to prove that a solution of \eqref{eq:model} dissipates the total free energy in time we need the following lemma.
\begin{lemma}\label{le:lemma1}
The following relation holds
\beq\label{eq:lemma1}
\nabla\cdot(C_0\nabla\phi\otimes\nabla\phi) = -\mu\nabla\phi +
\nabla\left(\frac{C_0}2|\nabla\phi|^2 + F(\phi)\right)\,.
\eeq
\end{lemma}
\begin{proof}
For $e_i$ the $i$-th unit vector and $d$ the spatial dimension the following relation holds
\beq
\ba{l}\dis
\nabla\cdot(C_0\nabla\phi\otimes\nabla\phi)
=C_0\sum_{j=1}^d\frac{\partial}{\partial x_j}\left(\sum_{i=1}^d\frac{\partial\phi}{\partial x_i}e_i\cdot\sum_{j=1}^d\frac{\partial\phi}{\partial x_j}e_j^T\right)
\\ \hueco\dis
=C_0\sum_{i,j=1}^d\frac{\partial}{\partial x_j}\left(\frac{\partial\phi}{\partial x_i}\frac{\partial\phi}{\partial x_j}\right)e_i
=C_0\sum_{i,j=1}^d\left(\frac{\partial\phi}{\partial x_i}\frac{\partial^2\phi}{\partial x_j^2}
+\frac{\partial^2 \phi}{\partial x_i\partial x_j}\frac{\partial\phi}{\partial x_j}\right)e_i
\\ \hueco\dis
=C_0\sum_{j=1}^d\frac{\partial^2 \phi}{\partial x_j^2}\cdot\sum_{i=1}^d\frac{\partial\phi}{\partial x_i}e_i
+\frac{C_0}2\sum_{i=1}^d\frac{\partial}{\partial x_i}\sum_{j=1}^d\left(\frac{\partial\phi}{\partial x_j}\right)^2 e_i
\\ \hueco\dis
=C_0\sum_{j=1}^d\frac{\partial^2 \phi}{\partial x_j^2}\cdot\sum_{i=1}^d\frac{\partial\phi}{\partial x_i}e_i
+\frac{C_0}2\sum_{i=1}^d\frac{\partial}{\partial x_i}\sum_{j=1}^d\left(\frac{\partial\phi}{\partial x_j}\right)^2 e_i
+ \left(\frac{\partial F(\phi)}{\partial\phi} - f(\phi) \right) \frac{\partial\phi}{\partial x_i}e_i
\\ \hueco\dis
=\left(C_0\sum_{j=1}^d\frac{\partial^2 \phi}{\partial x_j^2}-f(\phi)\right)\sum_{i=1}^d\frac{\partial\phi}{\partial x_i}e_i
+ \sum_{i=1}^d\frac{\partial}{\partial x_i}e_i\left(\frac{C_0}2\sum_{j=1}^d\left(\frac{\partial\phi}{\partial x_j}\right)^2 + F(\phi)\right)
\\ \hueco\dis
=\big(C_0\Delta\phi-f(\phi)\big)\nabla\phi + \nabla\left(\frac{C_0}2|\nabla\phi|^2 + F(\phi)\right)
\\ \hueco\dis
=-\mu\nabla\phi + \nabla\Big(\frac{C_0}2|\nabla\phi|^2 + F(\phi)\Big)\,.
\ea \nonumber
\eeq
\end{proof}
Now, we introduce the new pressure term $\tilde p = p + \frac{C_0}2|\nabla\phi|^2
+ F(\phi)$. Together with equation \eqref{eq:lemma1} this allows us to rewrite system \eqref{eq:model} as follows
\beq\label{eq:model3}
\bal
\frac{\partial\phi}{\partial t} + \u\cdot\nabla\phi &-
\nabla\cdot\left\{\phi(1-\phi)M\Big[\phi(1-\phi)\nabla\mu - \nabla (A_1(\phi)\,
q)\Big]\right\}=0\,,
\\ \hueco
\frac{\partial q}{\partial t} + \u\cdot\nabla q &+ \frac1{\tau_b(\phi)} q +
A_1(\phi)\nabla\cdot\left\{M\Big[\phi(1-\phi)\nabla\mu - \nabla (A_1(\phi)\,
q)\Big]\right\}=0\,,
\\ \hueco
\frac{\partial\sig}{\partial t} + (\u\cdot\nabla)\sig &-
(\nabla\u)\cdot\sig-\sig\cdot(\nabla\u)^T + \frac1{\tau_s(\phi) }\sig -
B_2(\phi)\Big[\nabla\u+(\nabla\u)^T\Big]=0\,,
\\ \hueco
\frac{\partial\u}{\partial t} + (\u\cdot\nabla)\u &- \nabla\cdot\left\{
\eta(\phi)\Big[\nabla\u+(\nabla\u)^T\Big]\right\} + \nabla \tilde p - \mu\nabla\phi -
\nabla\cdot\sig=0\,,
\\ \hueco
\nabla\cdot\u &= 0\,.
\eal
\eeq

\begin{theorem}\label{th:energylaw4}
System \eqref{eq:model3} obeys the following energy law
\begin{multline}\label{eq:energylaw4}
\frac{dE_{tot}(\phi,q,\sig,\u)}{dt} =
- \frac1{\tau^0_b} \left\|\frac{q}{\phi}\right\|^2_{L^2(\Omega)}
- \int_\Omega M\Big[\phi(1-\phi)\nabla\mu - \nabla (A_1(\phi)\, q)\Big]^2
\\ \hueco 
- \int_\Omega\frac1{2\,\tau_s(\phi) } \tr(\sig)
- \int_\Omega\frac{\eta(\phi)}2\sum_{i,j=1}^d\left( \frac{\partial u_i}{\partial x_j} + \frac{\partial u_j}{\partial x_i}\right)^2\,.
\end{multline}
\end{theorem}

\begin{proof}
Analogously to the derivation of the energy law~\eqref{eq:energylaw}, $\eqref{eq:model3}_1$ is multiplied
by $\mu$ and \eqref{eq:model3}$_2$ by $q$ and both are integrated. Assuming
suitable boundary conditions we obtain
\begin{multline*}
\frac{d E_{mix}(\phi)}{dt} + \frac{d E_{conf}(q)}{dt}
+ \frac1{\tau^0_b} \left\|\frac{q}{\phi}\right\|^2_{L^2(\Omega)} 
\\
+ \int_\Omega M\Big[\phi(1-\phi)\nabla\mu - \nabla (A_1(\phi)\, q)\Big]^2
+ \int_\Omega\u\cdot\nabla\phi\mu =0\,.
\end{multline*}
Further, we multiply \eqref{eq:model3}$_3$ by $\frac12\one$, where $\one$ is the unit matrix, and integrate over $\Omega$.
Taking into account that for all $\textbf{A} \in \R^{n\times n}, n\in\N, \textbf{A} : \one = \tr(\textbf{A}\cdot \one) = \tr(\textbf{A})$ we get 
\begin{align*}
& \int_\Omega \frac12 \tr\left(\frac{\partial\sig}{\partial t}\right) +
\int_\Omega \frac12 \tr\left((\u\cdot\nabla)\sig\right) -
 \int_\Omega\frac12 \left(\nabla\u:\sig^T + \sig:\nabla\u\right) 
 \\ \hueco
& -
\int_\Omega\frac1{2\,\tau_s(\phi) } \tr(\sig) +
B_2(\phi)\tr(\nabla\u)
\\ \hueco
= \quad & \int_\Omega \frac12 \frac{\partial\,\tr(\sig)}{\partial t} - \int_\Omega \frac12
\tr\left((\nabla\cdot\u)\sig\right) - \int_\Omega\sig:\nabla\u 
\\ \hueco
& +  \int_\Omega\frac1{2\,\tau_s(\phi) } \tr(\sig) + \int_\Omega B_2(\phi) (\nabla\cdot\u)
\\ \hueco
= \quad & \frac{d}{dt}\int_\Omega \frac12 \tr(\sig) - \int_\Omega\sig:\nabla\u +
\int_\Omega\frac1{2\,\tau_s(\phi) } \tr(\sig) = 0\,.
\end{align*}
Multiplying $\eqref{eq:model3}_4$ by $\u$ and integrating over $\Omega$ yields
\begin{align*}
& \int_\Omega \frac{\partial\u}{\partial t}\cdot\u
+ \int_\Omega(\u\cdot\nabla)|\u|^2
+ \int_\Omega\left\{\eta(\phi)\Big[\nabla\u+(\nabla\u)^T\Big]\right\}:\nabla\u
\\ \hueco 
& - \int_\Omega \tilde p(\nabla\cdot\u)
- \mu\nabla\phi\cdot\u
+ \nabla\cdot\sig\cdot\u
\\ \hueco 
= \quad & \int_\Omega \frac12\frac{\partial|\u|^2}{\partial t}
- \int_\Omega(\nabla\cdot\u)|\u|^2
+ \int_\Omega\left\{\eta(\phi)\Big[|\nabla\u|^2+Tr\left((\nabla\u)^2\right)\Big]\right\}
\\ \hueco 
& - \int_\Omega\u\cdot\nabla\phi\mu
+ \int_\Omega\sig:\nabla\u
\\ \hueco
= \quad & \frac{d}{dt}\int_\Omega \frac12|\u|^2
+ \int_\Omega\frac{\eta(\phi)}2\sum_{i,j=1}^d\left( \frac{\partial u_i}{\partial x_j} + \frac{\partial u_j}{\partial x_i}\right)^2
- \int_\Omega\u\cdot\nabla\phi\mu + \int_\Omega\sig:\nabla\u
=0\,.
\end{align*}
Summing up the above relations we obtain the energy law \eqref{eq:energylaw4}.
\end{proof}

\begin{remark}
We should point out that the elastic stress tensor $\sig$ does not necessarily need to be positive definite.
Thus, $tr(\sig)$ in the energy law~\eqref{eq:energylaw4} is not necessarily positive and could therefore interfere with the energy dissipation.
To control this we introduce the so-called conformation tensor $\c$, $\c :=\frac{1}{B_2(\phi)} \sig + \one$, where $\one$ is the identity matrix.
By its definition the conformation tensor $\c$ is positive definite. In \cite{hu} \textit{Hu and Leli\`evre} studied the classical Oldroyd-B model with $B_2(\phi)=$ const. and $\tau_s(\phi)=$ const. They were able to prove that
if the determinant of the initial conformation tensor is greater than one, then
$\tr(\sig)>0$ for all times. This result indicates that it is important to control the initial data for the elastic stress $\sig$ in such a way that the determinant of $\c$ is enough large in order to get an elastic stress tensor which remains positive definite as well. 
\end{remark}

\begin{remark} 
\change{In Subsection~\ref{sec:splitting} we will present a new splitting scheme. To this end we have to
introduce a new pressure term $\hat{p} = \tilde p - \phi\mu$ which allows us to rewrite system \eqref{eq:model3} as follows}
{Since the full model is incompressible, it holds $\nabla\cdot (\u w) = \u\cdot\nabla w + \nabla\cdot\u w = \u\cdot\nabla w, \ \ w\in\{\phi,q\},$ for the advection terms, allowing us to rewrite model \eqref{eq:model3} as follows}
\beq\label{eq:model3b}
\ba{l}\dis
\frac{\partial\phi}{\partial t} + \nabla\cdot (\u\phi) -
\nabla\cdot\left\{\phi(1-\phi)M\Big[\phi(1-\phi)\nabla\mu - \nabla (A_1(\phi)\,
q)\Big]\right\}=0\,,
\\ \hueco\dis
\frac{\partial q}{\partial t} + \nabla\cdot (\u q) + \frac1{\tau_b(\phi)} q +
A_1(\phi)\nabla\cdot\left\{M\Big[\phi(1-\phi)\nabla\mu - \nabla (A_1(\phi)\,
q)\Big]\right\}=0\,,
\\ \hueco\dis
\frac{\partial\sig}{\partial t} + (\u\cdot\nabla)\sig -
(\nabla\u)\cdot\sig-\sig\cdot(\nabla\u)^T + \frac1{\tau_s(\phi) }\sig -
B_2(\phi)\Big[\nabla\u+(\nabla\u)^T\Big]=0\,,
\\ \hueco\dis
\frac{\partial\u}{\partial t} + (\u\cdot\nabla)\u - \nabla\cdot\left\{
\eta(\phi)\Big[\nabla\u+(\nabla\u)^T\Big]\right\} + \nabla \hat p + \phi \nabla\mu  -
\nabla\cdot\sig=0\,,
\\ \hueco\dis
\quad\nabla\cdot\u=0\,\change{.}{,}
\ea
\eeq
\change{Note that for the advection term of the Cahn-Hilliard equation \eqref{eq:model3b}$_1$ it holds $\nabla\cdot (\u\phi) = \u\cdot\nabla\phi + \nabla\cdot\u\phi = \u\cdot\nabla\phi$, thanks  the incompressibility condition $\eqref{eq:model3b}_5$.}
{where the pressure term $\hat{p} = \tilde p - \phi\mu$. 
This model will be useful for our new splitting scheme in Subsection~\ref{sec:splitting}.}
\end{remark}
In what follows we will write for the sake of simplicity $p$ instead of $\tilde p$ and $\hat p$.

\section{Numerical schemes}\label{sec:num}
\subsection{Schemes for the simplified model}

We start this Subsection proposing the one step numerical scheme for the simplified model (\ref{eq:simp_model}). We consider an uniform
partition of the time interval $[0,T]$ with a constant time step $\Delta t$. Given
$(\phi^n,q^n)$ from the previous time step we compute $(\phi^{n+1},q^{n+1})$
such that
\beq\label{eq:simp_scheme}
\bal
\frac{\phi^{n+1} - \phi^n}{\Delta t} & -
\nabla\cdot\left\{\phi^n(1-\phi^n)M\Big[\phi^n(1-\phi^n)\nabla\mu^{
n+\frac12} - \nabla (A_1(\phi^n) \, q^{n+\frac12})\Big]\right\}=0\,,
\\ \hueco
\frac{q^{n+1} - q^n}{\Delta t} & + \frac1{\tau_b(\phi^n)} q^{n+\frac12} 
\\
& \quad \ \ + A_1(\phi^n)\nabla\cdot\left\{M\Big[\phi^n(1-\phi^n)\nabla\mu^{n+\frac12}
- \nabla (A_1(\phi^n) \, q^{n+\frac12})\Big]\right\}=0\,,
\eal
\eeq
where
$$
\mu^{n+\frac12} = -C_0\Delta\phi^{n+\frac12} + f(\phi^{n+1},\phi^n)\,.
$$
Here we use the notations $\phi^{n+\frac12}:=\frac{\phi^{n+1} + \phi^n}2$ and $q^{n+\frac12}:=\frac{q^{n+1} + q^n}2$ that are the Crank-Nicolson-type approximations.
\begin{theorem} \label{th:simp_scheme}
Let $f(\phi^{n+1}, \phi^n)$ represent a suitable linearized approximation of $f(\phi) = F'(\phi)$. Then the resulting 
numerical scheme \eqref{eq:simp_scheme} is linear and satisfies the following discrete version of the energy law \eqref{eq:energylaw}
\begin{multline} \label{eq:discreteenergylaw}
\frac{E_{tot}(\phi^{n+1},q^{n+1}) - E_{tot}(\phi^{n},q^{n})}{\Delta t}
= - ND^{n+1}_{phobic}
- \frac1{\tau^0_b}\left\|\frac{q^{n+\frac12}}{\phi^n}\right\|^2_{L^2(\Omega)}
\\ \hueco
- \int_\Omega M\Big[\phi^n(1-\phi^n)\nabla\mu^{n+\frac12} - \nabla (A_1(\phi^n) \, q^{n+\frac12})\Big]^2
\,,
\end{multline}
where
$$
ND^{n+1}_{phobic}:=\int_\Omega f(\phi^{n+1},\phi^n)\frac{\phi^{n+1} -
\phi^n}{\Delta t}
- \int_\Omega \frac{F(\phi^{n+1}) - F(\phi^n)}{\Delta t}\,.
$$
Depending on the approximation considered for $f(\phi^{n+1},\phi^n)$, we obtain different numerical schemes with different discrete energy laws, see Remark~\ref{re:potential}.
\end{theorem}

\begin{proof}
It is clear that the proposed scheme is linear.
The discrete mixing and conformation energy can be derived following the same calculations presented in the proof of Theorem~\ref{th:energylaw}.
Multiplying \eqref{eq:simp_scheme}$_1$ by $\mu^{n+\frac12}$, integrating over $\Omega$ and applying suitable boundary conditions yields
\begin{multline*}
\frac{E_{mix}(\phi^{n+1}) - E_{mix}(\phi^{n}) }{\Delta t}
+ ND^{n+1}_{phobic}
\\ \hueco
+ \int_\Omega M\Big[\phi^n(1-\phi^n)\nabla\mu^{n+\frac12} - \nabla
(A_1(\phi^n) \, q^{n+\frac12})\Big]\Big[\phi^n(1-\phi^n)\nabla \mu^{n+\frac12}\Big] =
0\,.
\end{multline*}
Analogously,  multiplying \eqref{eq:simp_scheme}$_2$ by $q^{n+\frac12}$,  integrating over $\Omega$ while assuming suitable boundary conditions implies
\begin{multline*}
\frac{E_{conf}(q^{n+1}) - E_{conf}(q^n)}{\Delta t}
+ \frac1{\tau^0_b}\left\|\frac{q^{n+\frac12}}{\phi^n}\right\|^2_{L^2(\Omega)}
\\ \hueco
+ \int_{\Omega}M\Big[\phi^n(1-\phi^n)\nabla\mu^{n+\frac12} - \nabla
(A_1(\phi^n) \, q^{n+\frac12})\Big] \Big[ - \nabla
(A_1(\phi^n)q^{n+\frac12})\Big] =0\,.
\end{multline*}
Summing up both relations leads to the discrete energy conservation law \eqref{eq:discreteenergylaw}.
\end{proof}


The numerical scheme \eqref{eq:simp_scheme} is linear and first order in time. We propose a linear second order numerical scheme by using a
second order extrapolation for the explicit terms, arriving at a two-step numerical scheme, obeying an analogous discrete energy law as scheme \eqref{eq:simp_scheme}.
The proposed linear second order numerical scheme reads
\beq\bal\label{eq:simp_scheme2}
&\frac{\phi^{n+1} - \phi^n}{\Delta t} 
\\
-&\nabla\cdot \left\{\phi^{n-\frac12}(1-\phi^{n-\frac12})M\Big[\phi^{n-\frac12}(1-\phi^{n-\frac12})\nabla\mu^{n+\frac12}  - \nabla (A_1(\phi^{n-\frac12}) \, q^{n+\frac12})\Big]\right\} =0\,,
\\ \hueco
& \frac{q^{n+1} - q^n}{\Delta t} + \frac1{\tau_b(\phi^{n-\frac12})} q^{n+\frac12}
\\
& \quad +A_1(\phi^{n-\frac12})\nabla\cdot\left\{M\Big[\phi^{n-\frac12}
(1-\phi^{n-\frac12})\nabla\mu^{n+\frac12} - \nabla (A_1(\phi^{n-\frac12}) \,
q^{n+\frac12})\Big]\right\}=0\,,
\eal\eeq
where
$$
\mu^{n+\frac12} = -C_0\Delta\phi^{n+\frac12} + f(\phi^{n+1},\phi^n)\,,
$$
and
$$
\phi^{n-\frac12} := \frac{3\phi^n-\phi^{n-1}}2
$$
is the second order extrapolation at the intermediate old time level $t_{n-1/2}$.
In order to compute the pair $(\phi^1,q^1)$ from $(\phi^0,q^0)$ a second order one-step nonlinear scheme could be considered.
We overcome this by setting $\phi^{-1}:=\phi^0$ and thus solving the first order scheme in the first time step.
As long as the initial data is sufficiently smooth, the influence is usually negligible for $T\gg 0$,
see the experimental order of convergence (EOC) presented in Section~\ref{sec:simulations}, Table~\ref{ta:eoc}.

\begin{definition}
A numerical scheme is called energy-stable if for any $n \in \N$ 
$$ 
E_{tot}(w^{n+1}) \le E_{tot}(w^n)\,, 
$$
where $w^{n+1}$ and $w^n$ are the solution vectors at times $t_{n+1}$ and $t_n$, respectively.
\end{definition}

\begin{remark}\label{re:potential}
  The choice of the approximation for $f(\phi^{n+1},\phi^n)$ strongly
  depends on the given potential $F(\phi)$. Anyway, in order to
  obtain an energy-stable numerical scheme it is necessary that
  $ND^{n+1}_{phobic}\ge 0$, i.~e.
$$
\int_\Omega f(\phi^{n+1},\phi^n)\frac{\phi^{n+1} -
\phi^n}{\Delta t}
\ge \int_\Omega \frac{F(\phi^{n+1}) - F(\phi^n)}{\Delta t}\,.
$$
In the literature the Ginzburg-Landau potential $F_{pol}$
\eqref{GL_pot} is often used, however with the above-mentioned
modification proposed, e.~g., by \textit{Wu, van Zwieten and van der Zee}~\cite{convex}. 
This potential $\tilde{F}_{pol}$ \eqref{eq:GL_mod} has a bounded second derivative,
$$
  \| f' \|_{L^\infty(\R)} = 
  \| f' \|_{L^\infty(-1,1)} = 2 .
$$
For this case we propose to use, following
\textit{Guill{\'e}n-Gonz{\'a}lez, Rodr{\'i}guez-Bellido and Tierra}~\cite{nematic}, 
the linear first order approximation
\beq\label{eq:pol1}
f_1(\phi^{n+1},\phi^n)=f(\phi^n)+\frac12||f'||_{L^\infty(\R)} (\phi^{n+1}-\phi^n) = f(\phi^n)+\phi^{n+1}-\phi^n\,,
\eeq
which has been shown~\cite{nematic} to satisfy $ND^{n+1}_{phobic}\ge 0$.

For the linear second order approximation we suggest to use the
convex-concave splitting of the potential, which has been proposed by
\textit{Wu et al.}~\cite{convex}. The
corresponding approximation for $f(\phi^{n+1},\phi^n)$ consists
of two second order Taylor approximations and reads
\beq\bal\label{eq:pol2}
f_2(\phi^{n+1},\phi^n)= & f_{vex}(\phi^{n+1}) - \frac{\phi^{n+1}-\phi^n}2 f'_{vex}(\phi^{n+1}) 
\\ \hueco
& +f_{cave}(\phi^n) + \frac{\phi^{n+1}-\phi^n}2 f'_{cave}(\phi^n)\,.
\eal\eeq
Since the convex part \change{}{reads} $f_{vex}=2\phi$, \change{}{approximation} (\ref{eq:pol2}) is linear. Note
that the derivative of the concave part, $f_{cave}=\phi^3-3\phi$, is
nonlinear and thus calculated explicitly. Further, to achieve energy
stability by using approximation \eqref{eq:pol2}, the chemical
potential has to be modified in the following way
\beq\bal\label{eq:modchem}
\mu^{n+\frac12} = &-C_0\Delta\phi^{n+\frac12} +
f_2(\phi^{n+1},\phi^n) 
\\ \hueco
&- \Delta t \frac{\left(\|f'_{vex}\|_{L^\infty(-1,1)}
+\|-f'_{cave}\|_{L^\infty(-1,1)}\right)^2}{16}
\Delta\phi^{n+\frac12} \,.
\eal\eeq

We note that the Flory-Huggins potential $F_{log}$ is logarithmic and
its derivatives are unbounded. Consequently, the choice of linear
approximations for $f(\phi^{n+1},\phi^n)$ while utilizing this more 
accurate potential is severely limited.
To ensure energy dissipation without modifying the potential, it is
necessary to use a nonlinear approximation for
$f(\phi^{n+1},\phi^n)$.  Since we are focusing on linear
schemes we propose to use the second order Taylor approximation
\beq\label{eq:OD2}
f_3(\phi^{n+1},\phi^n) = f(\phi^n) + \frac{\phi^{n+1}-\phi^n}2 f'(\phi^n)\,.
\eeq
This is called the ``optimal dissipation 2'' (OD2) approximation, see \cite{GT1},
because it leads to $ND^{n+1}_{phobic}=\mathcal{O}(\Delta t^2)$.
However, using this approximation it is not possible to control the sign of $ND^{n+1}_{phobic}$. Nevertheless, our numerical simulations presented in
Section~\ref{sec:simulations} suggest that the dissipation of the total energy is not violated.

In the recent work \cite{flory} \textit{Yang and Zhao} have proposed a modification of $F_{log}$ introducing a suitable cut off function close  to the boundaries in order to achieve the boundedness of the derivatives. 
Consequently, we can, e.~g., use the above mentioned convex-concave splitting \eqref{eq:pol2} with a modified chemical potential \eqref{eq:modchem}.
We may thus achieve a linear, second order and provably energy-stable numerical scheme, using a modified Flory-Huggins potential. Verification of this question is left for a future work.
\end{remark}

\subsection{Coupled schemes for \change{two-phase viscoelastic flows}{the full model}}\label{sec:coupled}

In this Subsection we present fully coupled linear energy dissipative schemes for the full \change{two-phase viscoelastic model \eqref{eq:model}.}{two-phase model for viscoelastic phase separation \eqref{eq:model}.}
Given \\ $(\phi^n, q^n, \sig^n, \u^n)$ from the previous time step we compute \\
$(\phi^{n+1}, q^{n+1}, \sig^{n+1}, \u^{n+1}, p^{n+1})$ such that
\beq\label{eq:scheme3}
\bal
\frac{\phi^{n+1} - \phi^n}{\Delta t} &+ \u^{n+1}\cdot\nabla\phi^n 
\\ 
& - \nabla\cdot\left\{\phi^n(1-\phi^n)M\Big[\phi^n(1-\phi^n)\nabla\mu^{
n+\frac12} - \nabla (A_1(\phi^n) \, q^{n+\frac12})\Big]\right\}=0\,,
\\ \hueco
\frac{q^{n+1} - q^n}{\Delta t} &+ \u^{n}\cdot\nabla q^{n+\frac12}
+ \frac1{\tau_b(\phi^n)} q^{n+\frac12}
\\ 
&+ A_1(\phi^n)\nabla\cdot\left\{M\Big[\phi^n(1-\phi^n)\nabla\mu^{n+\frac12}
- \nabla (A_1(\phi^n) \, q^{n+\frac12})\Big]\right\}=0\,,
\\ \hueco
\frac{\sig^{n+1} - \sig^n}{\Delta t} &+ (\u^{n}\cdot\nabla)\sig^n -
(\nabla\u^{n+1})\cdot\sig^n-\sig^n\cdot\left(\nabla\u^{n+1}\right)^T
\\ 
&+ \frac1{\tau_s(\phi^{n+\frac12})}\sig^n -
B_2(\phi^{n+\frac12})\Big[\nabla\u^{n+1}+(\nabla\u^{n+1})^T\Big]=0\,,
\\ \hueco
\frac{\u^{n+1} - \u^n}{\Delta t} &+ (\u^n\cdot\nabla)\u^{n+1} -
\nabla\cdot\left\{\eta(\phi^n)\Big[\nabla\u^{n+1}+(\nabla\u^{n+1})^T\Big]\right\} 
\\ \
&+ \nabla p^{n+1} - \mu^{n+\frac12}\nabla\phi^{n} -
\nabla\cdot\sig^n=0\,,
\\ \hueco
\nabla\cdot\u^{n+1} &= 0\,,
\eal
\eeq
where
$$
\mu^{n+\frac12} = -C_0\Delta\phi^{n+\frac12} + f(\phi^{n+1},\phi^n)\,.
$$
\begin{theorem}
The numerical scheme \eqref{eq:scheme3} is linear (up to the approximation
considered for $f(\phi^{n+1},\phi^n)$) and satisfies the discrete version of the energy law \eqref{eq:energylaw4}
\begin{multline}\label{eq:energylaw5}
\frac{E_{tot}(\phi^{n+1},q^{n+1},\sig^{n+1},\u^{n+1})-E_{tot}(\phi^{n},q^{n},\sig^n,\u^n)}{\Delta t}
= - ND^{n+1}_{phobic}
\\ \hueco
- \int_\Omega M\Big[\phi^n(1-\phi^n)\nabla\mu^{n+\frac12} - \nabla (A_1(\phi^n) \, q^{n+\frac12})\Big]^2
- \frac1{\tau^0_b}\left\|\frac{q^{n+\frac12}}{\phi^n}\right\|^2_{L^2(\Omega)}
\\ \hueco
- \int_\Omega\frac1{2\,\tau_s(\phi^{n+\frac12})} \tr(\sig^n)
- \frac1{2\,\Delta t}\|\u^{n+1} - \u^n\|^2_{L^2(\Omega)}
\\ \hueco
- \int_\Omega\frac{\eta(\phi^n)}2\sum_{i,j=1}^d\left( \frac{\partial u^{n+1}_i}{\partial x_j} + \frac{\partial u^{n+1}_j}{\partial x_i}\right)^2 
\,,
\end{multline}
where
$$
ND^{n+1}_{phobic}:=\int_\Omega f(\phi^{n+1},\phi^n)\frac{\phi^{n+1} -
\phi^n}{\Delta t}
- \int_\Omega \frac{F(\phi^{n+1}) - F(\phi^n)}{\Delta t}\,.
$$
\end{theorem}

\begin{proof}
Analogously to the derivation of the discrete energy law \eqref{eq:discreteenergylaw}, \eqref{eq:scheme3}$_1$ is multiplied by $\mu^{n+\frac12}$ and  \eqref{eq:scheme3}$_2$
by $q^{n+\frac12}$.  Integrating both equations over $\Omega$ and summing them up we obtain
\begin{multline*}
\frac{E_{mix}(\phi^{n+1}) - E_{mix}(\phi^n)}{\Delta t} + \frac{E_{conf}(q^{n+1}) - E_{conf}(q^n)}{\Delta t}
+ ND^{n+1}_{phobic}
\\ \hueco
+ \frac1{\tau^0_b}\left\|\frac{q^{n+\frac12}}{\phi^n}\right\|^2_{L^2(\Omega)}
+ \int_{\Omega}M\Big[\phi^n(1-\phi^n)\nabla\mu^{n+\frac12}-\nabla (A_1(\phi^n) \, q^{n+\frac12})\Big]^2
\\ \hueco
+ \int_\Omega\u^{n+1}\cdot\nabla\phi^n\mu^{n+\frac12} = 0\,.
\end{multline*}
Further, we multiply \eqref{eq:scheme3}$_3$ by $\frac12\one$ and apply analogous calculations as for the shear stress part of the continuous energy law~(\ref{eq:energylaw4})
$$
\frac{E_{el}(\sig^{n+1}) - E_{el}(\sig^n)}{\Delta t} - \int_\Omega\sig^n:\nabla\u^{n+1} +
\int_\Omega\frac1{2\,\tau_s(\phi^{n+\frac12}) } \tr(\sig^n) = 0\,.
$$
Multiplying \eqref{eq:scheme3}$_4$ by $\u^{n+1}$ and integrating over $\Omega$ leads to
\begin{multline*}
\int_\Omega \frac{\u^{n+1} - \u^n}{\Delta t} \cdot\u^{n+1}
+ \int_\Omega\frac{\eta(\phi^n)}2\sum_{i,j=1}^d\left( \frac{\partial
u^{n+1}_i}{\partial x_j} + \frac{\partial u^{n+1}_j}{\partial x_i}\right)^2
\\ \hueco
- \int_\Omega\u^{n+1}\cdot\nabla\phi^n\mu^{n+\frac12}
+ \int_\Omega\sig^n:\nabla\u^{n+1}=0\,,
\end{multline*}
where
\begin{align*}
\int_\Omega \frac{\u^{n+1} - \u^n}{\Delta t} \cdot\u^{n+1}
&= \frac1{\Delta t} \int_\Omega |\u^{n+1}|^2 - \u^n\cdot\u^{n+1}
\\ \hueco
&= \frac1{2\,\Delta t} \int_\Omega |\u^{n+1}-\u^n|^2 + |\u^{n+1}|^2 - |\u^n|^2
\\ \hueco
&= \frac1{2\,\Delta t}\|\u^{n+1} - \u^n\|^2_{L^2(\Omega)}
+ \frac{E_{kin}(\u^{n+1})-E_{kin}(\u^n)}{\Delta t}\,.
\end{align*}
The discrete energy law \eqref{eq:energylaw5} is achieved by summing the above relations.
\end{proof}


\begin{remark}
It is possible to eliminate the term $\frac1{2\,\Delta t}\|\u^{n+1} -
\u^n\|^2_{L^2(\Omega)}$ from the energy law, considering the following linear
one-step scheme.
\beq\label{eq:scheme3b}
\bal
\frac{\phi^{n+1} - \phi^n}{\Delta t} &+ \u^{n+\frac12}\cdot\nabla\phi^n
\\
&- \nabla\cdot\left\{\phi^n(1-\phi^n)M\Big[\phi^n(1-\phi^n)\nabla\mu^{n+\frac12}-\nabla (A_1(\phi^n) \, q^{n+\frac12})\Big]\right\}=0\,,
\\ \hueco
\frac{q^{n+1} - q^n}{\Delta t} &+ \u^{n}\cdot\nabla q^{n+\frac12}
+ \frac1{\tau_b(\phi^n)} q^{n+\frac12} 
\\ 
&+ A_1(\phi^n)\nabla\cdot\left\{M\Big[\phi^n(1-\phi^n)\nabla\mu^{n+\frac12}
- \nabla (A_1(\phi^n) \, q^{n+\frac12})\Big]\right\} = 0\,,
\\ \hueco
\frac{\sig^{n+1} - \sig^n}{\Delta t} &+ (\u^{n+\frac12}\cdot\nabla)\sig^n
- (\nabla\u^{n+\frac12})\cdot\sig^n-\sig^n\cdot\left(\nabla\u^{n+\frac12}\right)^T
\\
&\qquad \qquad \qquad 
+ \frac1{\tau_s(\phi^{n+\frac12})}\sig^n
- B_2(\phi^{n+\frac12})2D(\u^{n+\frac12})=0\,,
\\ \hueco
\frac{\u^{n+1} - \u^n}{\Delta t} &+ (\u^n\cdot\nabla)\u^{n+\frac12}
- \nabla\cdot\left\{\eta(\phi^n)2D(\u^{n+\frac12})\right\}
+ \nabla p^{n+\frac12}
\\ 
&\qquad \qquad \qquad \qquad \qquad \qquad
- \mu^{n+\frac12}\nabla\phi^{n}
- \nabla\cdot\sig^n=0\,,
\\ \hueco
\nabla\cdot\u^{n+\frac12} &= 0\,,
\eal
\eeq
where
$$
\mu^{n+\frac12} = -C_0\Delta\phi^{n+\frac12} + f(\phi^{n+1},\phi^n)\,,
$$
and
$$
D(\u^{n+\frac12}) = \frac12 \Big[\nabla\u^{n+\frac12}+(\nabla\u^{n+\frac12})^T\Big]\,.
$$
\end{remark}

Analogous to scheme~\eqref{eq:simp_scheme2}, using the second order extrapolation $z^{n-\frac12}=\frac{3z^n-z^{n-1}}2, z\in\{\phi, \u, \sig\},$
for the explicit terms in scheme \eqref{eq:scheme3b} yields the following linear and
second order in time two-step numerical scheme
\beq\label{eq:scheme3c}
\bal
&\frac{\phi^{n+1} - \phi^n}{\Delta t} + \u^{n+\frac12}\cdot\nabla\phi^{n-\frac12}
\\
-& \nabla\cdot\left\{\phi^{n-\frac12}(1-\phi^{n-\frac12})M\Big[\phi^{n-\frac12}(1-\phi^{n-\frac12})\nabla\mu^{n+\frac12} - \nabla (A_1(\phi^{n-\frac12}) \, q^{n+\frac12})\Big]\right\} = 0\,,
\\ \hueco
&\frac{q^{n+1} - q^n}{\Delta t} + \u^{n-\frac12}\cdot\nabla q^{n+\frac12}
+ \frac1{\tau_b(\phi^{n-\frac12})} q^{n+\frac12}
\\
&\quad + A_1(\phi^{n-\frac12})\nabla\cdot\left\{M\Big[\phi^{n-\frac12}(1-\phi^{n-\frac12})\nabla\mu^{n+\frac12}
- \nabla (A_1(\phi^{n-\frac12}) \, q^{n+\frac12})\Big]\right\} = 0\,,
\\ \hueco
&\frac{\sig^{n+1} - \sig^n}{\Delta t} + (\u^{n+\frac12}\cdot\nabla)\sig^{n-\frac12}
- (\nabla\u^{n+\frac12})\cdot\sig^{n-\frac12}-\sig^{n-\frac12}\cdot\left(\nabla\u^{n+\frac12}\right)^T
\\
& \qquad \qquad \qquad \qquad \qquad \qquad + \frac1{\tau_s(\phi^{n+\frac12})}\sig^{n-\frac12}
- B_2(\phi^{n+\frac12})2D(\u^{n+\frac12}) = 0\,,
\\ \hueco
&\frac{\u^{n+1} - \u^n}{\Delta t} + (\u^{n-\frac12}\cdot\nabla)\u^{n+\frac12} 
\\
& \quad - \nabla\cdot \left\{\eta(\phi^{n-\frac12})2D(\u^{n+\frac12})\right\}
+ \nabla p^{n+\frac12}
- \mu^{n+\frac12}\nabla\phi^{n-\frac12}
- \nabla\cdot\sig^{n-\frac12} = 0\,,
\\ \hueco
& \quad\nabla\cdot\u^{n+\frac12} = 0\,.
\eal
\eeq
Scheme \eqref{eq:scheme3c} satisfies an analogous discrete energy law as scheme \eqref{eq:scheme3b}.

\begin{remark}\label{re:stability}
Note that for small shear rates $D(u)$ and the Weissenberg numbers $\tau^0_s$ that typically arise in our numerical experiments, 
the stiffness of the Oldroyd-B equation does not play a dominant role.
If it is required the high Weissenberg problem can be treated by using additional techniques like the logarithmic transformation of the conformation tensor or considering the stress diffusion term in the evolution equation for $\sig$,
for more details see, e.~g., \textit{Luk\'a\v{c}ov\'a-Medvi{\softd}ov\'a, Notsu, and She}~\cite{bangwei}.
For large shear rates an implicit approximation of the elastic shear stress is suitable, but it hurts the linearity  of a numerical scheme. The proposed modification of scheme \eqref{eq:scheme3b} reads
\beq\label{eq:scheme4_2}
\bal
\frac{\phi^{n+1} - \phi^n}{\Delta t} &+ \u^{n+1}\cdot\nabla\phi^n
\\ \hueco
& - \nabla\cdot\left\{\phi^n(1-\phi^n)M\Big[\phi^n(1-\phi^n)\nabla\mu^{n+\frac12}-\nabla (A_1(\phi^n) \, q^{n+\frac12})\Big]\right\}=0\,,
\\ \hueco
\frac{q^{n+1} - q^n}{\Delta t} &+ \u^{n}\cdot\nabla q^{n+\frac12} + \frac1{\tau_b(\phi^n)} q^{n+\frac12}
\\ \hueco&
+ A_1(\phi^n)\nabla\cdot\left\{M\Big[\phi^n(1-\phi^n)\nabla\mu^{n+\frac12}-\nabla (A_1(\phi^n) \, q^{n+\frac12})\Big]\right\}=0\,,
\\ \hueco
\frac{\sig^{n+1} - \sig^n}{\Delta t} &+ (\u^{n+\frac12}\cdot\nabla)\sig^{n+1}
- (\nabla\u^{n+\frac12})\cdot\sig^{n+1} - \sig^{n+1}\cdot\left(\nabla\u^{n+\frac12}\right)^T
\\ \hueco&
+ \frac1{\tau_s(\phi^n)}\sig^{n+1} + B_2(\phi^{n+\frac12})2D(\u^{n+\frac12})=0\,,
\\ \hueco
\frac{\u^{n+1} - \u^n}{\Delta t} &+ (\u^n\cdot\nabla)\u^{n+\frac12}
- \nabla\cdot\left\{\eta(\phi^n)\,2D(\u^{n+\frac12})\right\}
+ \nabla p^{n+\frac12}
\\ \hueco &
- \mu^{n+\frac12}\nabla\phi^{n}
-\nabla\cdot\sig^{n+1} =0\,,
\\ \hueco
\nabla\cdot\u^{n+\frac12} &= 0\,.
\eal
\eeq
Scheme \eqref{eq:scheme4_2} also satisfies an analogous discrete energy law as scheme \eqref{eq:scheme3b}.

Note that we can linearize scheme \eqref{eq:scheme4_2} by, e.~g., using a fixed point iteration, see Remark~\ref{re:fixpoint}.
Further, using the idea presented in scheme~\eqref{eq:scheme3c} concerning the extrapolation of  the explicit terms, while utilizing the Crank-Nicolson-type approximation $\sig^{n+\frac12}$ for the implicit terms,
we can obtain a second order  two-step numerical scheme.
\end{remark}

\subsection{Splitting scheme for \change{two-phase viscoelastic flows}{the full model}}\label{sec:splitting}

In this Subsection we present yet another possibility to discretize system \eqref{eq:model3b}. In order to save computational costs we
split the computation into three different substeps. The first two steps are the interesting ones allowing us to decouple the calculation of the fluid part $(\u,p)$
from the phase field and bulk stress parts $(\phi,q)$. The third step is the calculation of the shear stress part $\sig$. 
In the first step we discretize the simplified model.

\textbf{Step 1.} Find $(\phi^{n+1},q^{n+1})$ such that
\beq\label{eq:scheme1_step1}
\bal
\frac{\phi^{n+1} - \phi^n}{\Delta t} &+ \nabla\cdot (\u^*\phi^n)
\\ \hueco &
- \nabla\cdot\left\{\phi^n(1-\phi^n)M\Big[\phi^n(1-\phi^n)\nabla\mu^{n+\frac12}
- \nabla (A_1(\phi^n) \, q^{n+\frac12})\Big]\right\} = 0\,,
\\ \hueco
\frac{q^{n+1} - q^n}{\Delta t} &+ \nabla\cdot (\u^{n} q^{n+\frac12})
+ \frac1{\tau_b(\phi^n)} q^{n+\frac12}
\\ &  
+ A_1(\phi^n)\nabla\cdot\left\{M\Big[\phi^n(1-\phi^n)\nabla\mu^{n+\frac12}-\nabla (A_1(\phi^n) \, q^{n+\frac12})\Big]\right\} = 0\,,
\eal
\eeq
where
$$
\mu^{n+\frac12} = -C_0\Delta\phi^{n+\frac12} + f(\phi^{n+1},\phi^n)\,,
$$
and
\beq\label{def:ustar}
\u^* := \u^n - \Delta t \,\phi^n\nabla\mu^{n+\frac12}\,,
\eeq
to split the phase field part from the hydrodynamic part.
In the second step we discretize the fluid equations as follows.

\textbf{Step 2.} Find $(\u^{n+1},p^{n+1})$ such that
\beq\label{eq:scheme1_step2}
\bal
\frac{\u^{n+1} - \u^*}{\Delta t} &+ (\u^n\cdot\nabla)\u^{n+1}
- \nabla\cdot\left\{\eta(\phi^{n+\frac12})\,2D(\u^{n+1})\right\}
+ \nabla p^{n+1}
- \nabla\cdot\sig^n
= 0\,,
\\ \hueco
\nabla\cdot\u^{n+1} &= 0\,.
\eal
\eeq
Finally, in the third step we approximate the Oldroyd-B equation for
the time evolution of the shear stress tensor $\sig$.

\textbf{Step 3.} Find $\sig^{n+1}$ such that
\beq\bal\label{eq:scheme1_step3}
\frac{\sig^{n+1} - \sig^n}{\Delta t} + (\u^{n+1}\cdot\nabla)\sig^n
&- (\nabla\u^{n+1})\cdot\sig^n - \sig^n\cdot\left(\nabla\u^{n+1}\right)^T
\\ \hueco &
+ \frac1{\tau_s(\phi^{n+\frac12})}\sig^n
- B_2(\phi^{n+\frac12})\,2D(\u^{n+1}) = 0\,.
\eal\eeq

\begin{theorem}
The numerical scheme
\eqref{eq:scheme1_step1}-\eqref{eq:scheme1_step3} is linear and satisfies the discrete energy law
\begin{multline}\label{eq:energylaw_split}
\frac{E_{tot}(\phi^{n+1},q^{n+1},\sig^{n+1},\u^{n+1}) -
E_{tot}(\phi^{n},q^{n},\sig^n,\u^n)}{\Delta t}
= - ND^{n+1}_{phobic}
- ND^{n+1}_{split}
\\ \hueco
- \int_\Omega M\Big[\phi^n(1-\phi^n)\nabla\mu^{n+\frac12}-\nabla (A_1(\phi^n) \, q^{n+\frac12})\Big]^2
- \frac1{\tau^0_b}\left\|\frac{q^{n+\frac12}}{\phi^n}\right\|^2_{{L^2(\Omega)}}
\\ \hueco
- \int_\Omega\frac1{2\,\tau_s(\phi^{n+\frac12})} \tr(\sig^n)
- \int_\Omega\frac{\eta(\phi^{n+\frac12})}2\sum_{i,j=1}^d\left(\frac{\partial u^{n+1}_i}{\partial x_j}+\frac{\partial u^{n+1}_j}{\partial x_i}\right)^2\,,
\end{multline}
where
%
$$
ND^{n+1}_{split}:=\frac1{2\,\Delta t}\Big(\|\u^{n+1} -
\u^*\|^2_{L^2(\Omega)} + \|\u^*
- \u^n\|^2_{L^2(\Omega)} \Big)\,.
$$
\end{theorem}

\begin{proof}
  Similar to the proof of the discrete energy law \eqref{eq:energylaw5} we multiply
  \eqref{eq:scheme1_step1}$_1$ by $\mu^{n+\frac12}$,
  \eqref{eq:scheme1_step1}$_2$ by $q^{n+\frac12}$,
  \eqref{eq:scheme1_step2}$_1$ by $\u^{n+1}$, and
  \eqref{eq:scheme1_step3} by $\frac12\one$, and integrate over $\Omega$.  Assuming
  suitable boundary conditions the derivations of the discrete
  elastic energies are analogous, while the calculation of the discrete mixing energy leads to the
  additional term $\int_\Omega \nabla\cdot (\u^*\phi^n) \mu^{n+\frac12}$.
  The key idea of the splitting scheme lies in  matching this term with $\int_\Omega \frac1{\Delta t}(\u^{n+1} - \u^*)\cdot\u^{n+1}$.
  This is possible by multiplying expression \eqref{def:ustar} by $\u^*$ and integrating over $\Omega$, which yields
  \begin{eqnarray*}
  \|\u^*\|^2_{L^2(\Omega)}
  &=&
  \int_\Omega \u^n\cdot\u^*  - \int_\Omega \Delta t \,\phi^n\nabla\mu^{n+\frac12} \cdot \u^* 
  \\ &=& 
  \int_\Omega \u^n\cdot\u^*  + \int_\Omega \Delta t \,\nabla\cdot (\u^*\phi^n) \mu^{n+\frac12}
  \,.
  \end{eqnarray*}
  This can be rewritten as follows
  \begin{eqnarray*}
    \int_\Omega \nabla\cdot (\u^*\phi^n) \mu^{n+\frac12}
    &=&
    \frac1{\Delta t}\left(\|\u^*\|^2_{L^2(\Omega)} - \int_\Omega
    \u^n\cdot\u^* \right) 
    \\ &=& 
    \frac1{2\,\Delta t}\Big(\|\u^*\|_{L^2(\Omega)}^2 -
    \|\u^n\|_{L^2(\Omega)}^2
    + \|\u^* - \u^n\|_{L^2(\Omega)}^2\Big)
    \,.
  \end{eqnarray*}
Now, since
  $$
  \int_\Omega \frac 1{\Delta t}(\u^{n+1} - \u^*)\cdot\u^{n+1}
  =\frac1{2\,\Delta t}\Big(\|\u^{n+1}\|_{L^2(\Omega)}^2 -
  \|\u^*\|_{L^2(\Omega)}^2
  + \|\u^{n+1} - \u^*\|_{L^2(\Omega)}^2\Big)\,,
  $$
the additional terms $\pm \frac 1{2\Delta t} \|\u^*\|_{L^2(\Omega)}^2 $ are canceled out and we obtain the desired discrete energy law \eqref{eq:energylaw_split}.    
\end{proof}
Consequently,
  \begin{equation} \label{discrete_dissipation}
    E_{tot}(\phi^{n+1},q^{n+1},\sig^{n+1},\u^{n+1})
    \leq E_{tot}(\phi^{n},q^{n},\sig^n,\u^n),
  \end{equation}
  provided that we control $ND^{n+1}_{phobic}$, since all other terms are non-negative.
  Indeed, even using the OD2 approximation~\eqref{eq:OD2},
  our numerical experiments in Chapter~\ref{sec:simulations} suggest that the energy
  dissipation (\ref{discrete_dissipation}) holds, see Figure~\ref{fig:en1}.
\begin{remark}\label{re:chorin}
To further reduce the computational costs of our splitting scheme in \textbf{Step 2}
we propose to use Chorin's projection method, see \textit{Chorin}~\cite{chorin}.  This well-known algorithm allows to decouple computation of the velocity  and the pressure of system \eqref{eq:scheme1_step2}.

\textbf{Step I.} Find $\u^{\dagger}$ such that
$$
\frac{\u^{\dagger} - \u^n}{\Delta t} + (\u^{n}\cdot\nabla)\u^{\dagger} -
\nabla\cdot\left\{\eta(\phi^n)\left[\nabla\u^{\dagger}+\left(\nabla\u^{\dagger}\right)^T\right]\right\} +
\phi^n\nabla\mu^{n+\frac12}
-\nabla \cdot \sig^n=0\,,
$$
and thus
\beq\label{eq:pressure}
\frac{\u^{n+1} - \u^{\dagger}}{\Delta t} = \nabla p^{n+1}\,.
\eeq

\textbf{Step II.} Applying the divergence to \eqref{eq:pressure} yields
$$
\frac{\nabla\cdot\u^{n+1} - \nabla\cdot\u^{\dagger}}{\Delta t} = \Delta p^{n+1}\,.
$$
Consequently, due to the incompressibility condition $\nabla\cdot \u^{n+1}=0$ we find $p^{n+1}$ such that
$$
\Delta p^{n+1} = - \frac{\nabla\cdot\u^{\dagger}}{\Delta t}\,.
$$

\textbf{Step III.} Since $\u^\dagger$ and $p^{n+1}$ are now known, we find $\u^{n+1}$ by solving \eqref{eq:pressure}.\\
In summary, instead of solving a coupled system for $(\u^{n+1},p^{n+1})$, we compute $(\u^{\dagger},p^{n+1},\u^{n+1})$ in a decoupled way.
\end{remark}

\begin{remark}\label{re:fixpoint}
For large shear rates $D(u)$ an implicit approximation of the shear stress would be suitable,
but it would hurt the linearity of the numerical scheme,  see also Remark~\ref{re:stability}.
The proposed modification of \eqref{eq:scheme1_step2} and \eqref{eq:scheme1_step3} reads

\textbf{Step 2$^*$.} Find $(\sig^{n+1},\u^{n+1},p^{n+1})$ such that
\beq\label{eq:scheme5_2_step2tilde}
\bal
\frac{\sig^{n+1} - \sig^n}{\Delta t} + (\u^{n+1}\cdot\nabla)\sig^{n+1}
&- (\nabla\u^{n+1})\cdot\sig^{n+1} - \sig^{n+1}\cdot\left(\nabla\u^{n+1}\right)^T
\\
&+ \frac1{\tau_s(\phi^{n+\frac12})}\sig^{n+1} +B_2(\phi^{n+\frac12})2D(\u^{n+1})
=0\,,
\\ \hueco
\frac{\u^{n+1} - \u^n}{\Delta t} + (\u^n\cdot\nabla)\u^{n+1}
&- \nabla\cdot\left\{\eta(\phi^{n+\frac12})\,2D(\u^{n+1})\right\}
+ \nabla p^{n+1}
\\ &
+ \phi^n\nabla\mu^{n+\frac12}
-\nabla\cdot\sig^{n+1} =0\,,
\\ \hueco
\nabla\cdot\u^{n+1} = 0\,. \qquad \qquad &
\eal
\eeq

It is possible to linearize and split \textbf{Step 2$^*$} again by, e.~g., using the following fixpoint iteration.
Given $(\sig^{n,0}=\sig^n,\u^{n,0}=\u^n)$ from the previous time step, we repeat \textbf{Step 2$^\dagger$} and
\textbf{3$^\dagger$} for $l=0,1,...$, until $||z^{n,l+1}-z^{n,l}||\le\delta||z^{n,l}||$, 
for $z \in \{\sig,\u,p\}$ and $\delta$ sufficiently small.

\textbf{Step 2$^\dagger$.} Find $(\u^{n,l+1},p^{n,l+1})$ such that
\beq\label{eq:scheme5_2_step2hat}
\bal
\frac{\u^{n,l+1} - \u^n}{\Delta t} + (\u^n\cdot\nabla)\u^{n,l+1}
&- \nabla\cdot\left\{\eta(\phi^{n+\frac12})\,2D(\u^{n,l+1})\right\}
+ \nabla p^{n,l+1}
\\ &
+ \phi^n\nabla\mu^{n+\frac12}
- \nabla\cdot(\sig^{n,l})=0\,,
\\ \hueco
\nabla\cdot\u^{n,l+1} = 0\,, \qquad \qquad & 
\eal
\eeq
where
$$
D(\u^{n,l+1}) = \frac12 \Big[ \nabla\u^{n,l+1}+(\nabla\u^{n,l+1})^T \Big]\,.
$$

\textbf{Step 3$^\dagger$.} Find $\sig^{n,l+1}$ such that
\begin{multline}\label{eq:scheme5_2_step3}
\frac{\sig^{n,l+1} - \sig^n}{\Delta t} + (\u^{n,l+1}\cdot\nabla)\sig^{n,l+1} -
(\nabla\u^{n,l+1})\cdot\sig^{n,l+1}-\sig^{n,l+1}\cdot\left(\nabla\u^{n,l+1}\right)^T
\\ \hueco 
+ \frac1{\tau_s(\phi^{n+\frac12})}\sig^{n,l+1} + B_2(\phi^{n+\frac12}) 2 D(\u^{n,l+1})
=0\,.
\end{multline}

\textbf{Step 4$^\dagger$.} Update solution: $\u^{n+1}=\u^{n,l+1},p^{n+1}=p^{n,l+1},\sig^{n+1}=\sig^{n,l+1}$.
Note that we can also use Chorin's projection method from Remark~\ref{re:chorin} in \textbf{Step 2$^\dagger$}.

\end{remark}
Let us point out that we have proven energy dissipation for
semi-discrete schemes. \change{The proofs of the fully discrete schemes are
analogous and our second order in space finite volume / finite
difference scheme preserves this property as well, see
\textit{Luk\'a\v{c}ov\'a-Medvi{\softd}ov\'a et al.}~\cite{bangwei}.}{The spatial discretization is done by the second order finite volume/ finite difference scheme.
The degrees of freedom for velocities are the centers of cell faces. For example, in two space dimensions the $x$-velocity component is given at the centers of vertical cell faces, 
while the $y$-velocity component is given at the centers of horizontal cell faces. Consequently, the velocity components are piecewise linear in one direction and constant in the other. 
Other variables $(\phi,q,\sig,p)$ are piecewise constant and given at the cell centers. This is analogous to the Marker and Cell (MAC) method by \textit{Harlow and Welch}~\cite{mac}.
We use an upwind finite volume method to approximate the advection terms and central finite differences for other derivatives. This discretization is mass conserving for $\phi$. 
The proof of the energy dissipation property for the fully discrete schemes can be done in an analogous way to that for the semi-discrete schemes, 
see \textit{Luk\'a\v{c}ov\'a-Medvi{\softd}ov\'a et al.}~\cite{bangwei}.}

\section{Numerical experiments}\label{sec:simulations}

In this Section we illustrate the behavior of the newly derived
numerical schemes in 2D. For the full model
\eqref{eq:model}/\eqref{eq:model3b} we apply the splitting scheme
\eqref{eq:scheme1_step1}-\eqref{eq:scheme1_step3}. Here we use the
Chorin projection method, see Remark~\ref{re:chorin}, and the optimal
dissipation 2 approximation \eqref{eq:OD2} for $f(\phi^{n+1},\phi^n)$, see
Remark~\ref{re:potential}, 
since we are utilizing the Flory-Huggins potential $F_{log}$ \eqref{eq:fh} in our model equations.

The simplified model \eqref{eq:simp_model} is simulated using the
second order scheme \eqref{eq:simp_scheme2} with the optimal
dissipation 2 approximation.  Note that for
our numerical schemes we can use larger $\Delta t$ than that applied
in \textit{Zhou et al.}~\cite{zhou}.  For example for our second order
scheme for the simplified model we can set $\Delta t = 0.25$ instead
of $\Delta t = 0.025$ as in \textit{Zhou et al.}~\cite{zhou}. This is
related to the fact that our energy dissipative schemes are more
stable.

We start with the numerical analysis of our numerical scheme
\eqref{eq:simp_scheme2} for the simplified model by calculating its
experimental order of convergence (EOC) in time.  Therefore the finest
resolution ($\Delta t=2^{-4}\cdot 10^{-3}$) is used as the reference
solution $z_{ref}$, $z\in\{\phi,q\}$, and
$$
EOC(z)=\log_2(e(z)/e(z'))\,,
$$
where $e(z)=\|z-z_{ref}\|_{L_1(\Omega)}$ and $e(z')$ are the
$L_1$-errors of two numerical solutions computed with the consecutive
time steps $\Delta t$ and $\Delta t' = \Delta t / 2$.
Table~\ref{ta:eoc} clearly indicates that our claim that scheme
\eqref{eq:simp_scheme2} is of second order in time is true.

\begin{table}
  \centering
  \caption{Experimental order of convergence (EOC) in time of
    scheme \eqref{eq:simp_scheme2} using the smooth initial data
    $\phi_0 = 0.5+0.5\sin(2\pi x)sin(2\pi y)$ and apart from that
    the parameters of the first numerical experiment.}
\begin{tabular}{|c | c|c |c |c|}
                \hline
                $\Delta t$/$\Delta t'$   & $L_1$-error $\phi$ & $EOC(\phi)$ & $L_1$-error $q$ & $EOC(q)$ \\
                \hline
                $8/4\cdot 10^{-3}$           & $3.765$              &        & 1.812              &       \\
                $4/2\cdot 10^{-3}$           & $1.069$              &  1.817 & 0.4613             & 1.974 \\
                $2/1\cdot 10^{-3}$           & $0.2824$             &  1.920 & 0.1172             & 1.977 \\
                $1/\frac12\cdot 10^{-3}$      & $7.124\cdot 10^{-2}$ &  $1.987$ & $2.907\cdot 10^{-2}$ & 2.012 \\
                $\frac12/\frac14\cdot 10^{-3}$ & $1.779\cdot 10^{-2}$ &  $2.002$ & $7.212\cdot 10^{-3}$ & 2.011 \\
                $\frac14/\frac18\cdot 10^{-3}$ & $4.803\cdot 10^{-3}$ &  $1.889$ & $1.854\cdot 10^{-3}$ & 1.960 \\
                $\frac18/\frac1{16}\cdot 10^{-3}$ & $1.462\cdot 10^{-3}$ &  $1.716$ & $4.401\cdot 10^{-4}$ & 2.075 \\
                \hline
\end{tabular}
\label{ta:eoc}
\end{table}

\begin{figure}[t]
  \centering
  \includegraphics[width=1\textwidth,trim=3.5cm 2cm 1.5cm 1cm]{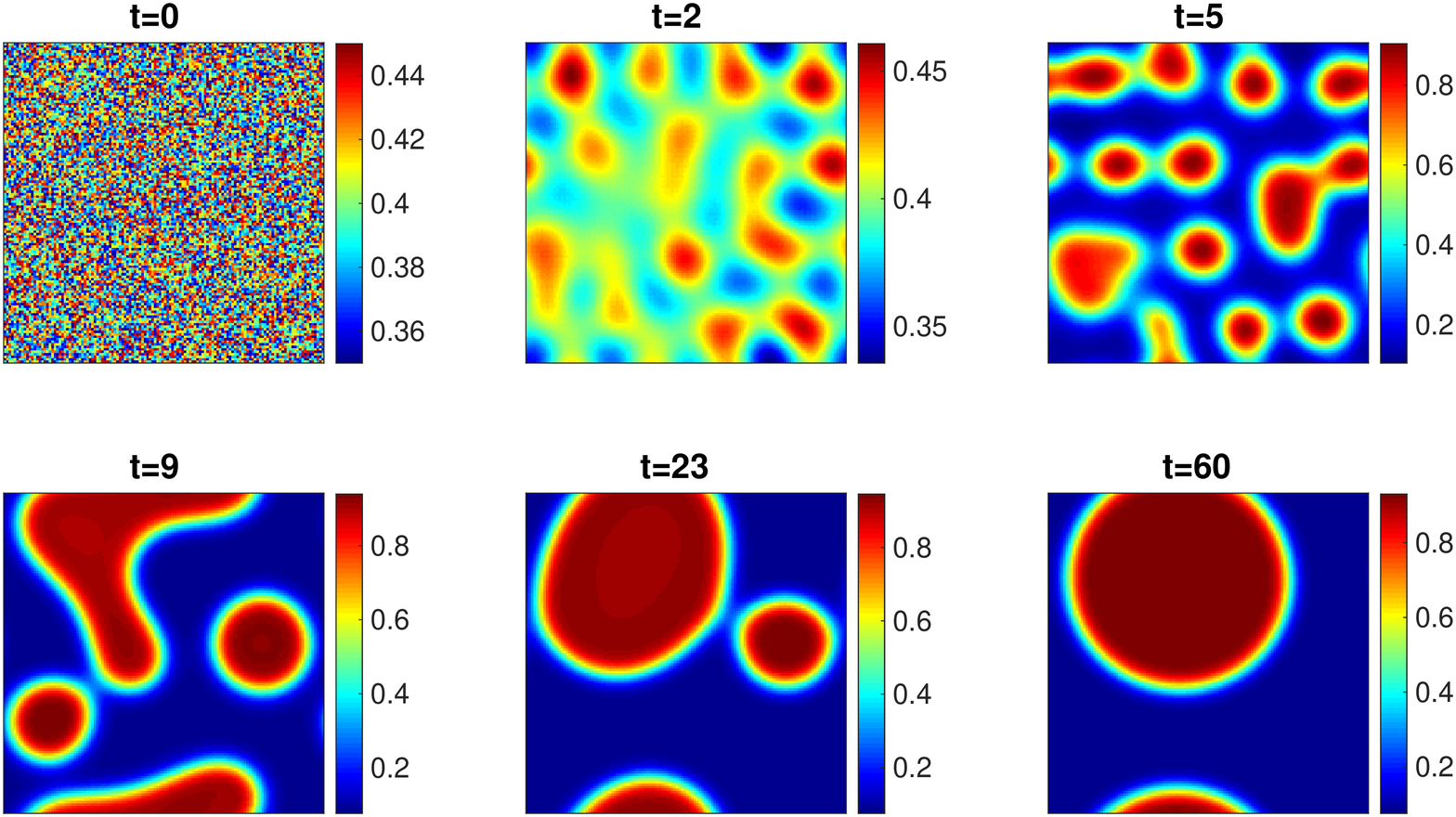}
  \caption{Time evolution of the polymer-solvent phase separation
    after a temperature quench with $\chi = 3, \phi_0 = 0.4$ and
    $\omega =[-0.05,0.05]$. The computational domain is
    $\Omega=[0,1]\times[0,1]$ and the interface width $\sqrt{C_0} =
    \frac1{\sqrt{600}}$. The time step is $\Delta t = 10^{-4}$.}
  \label{fig:ex1}
\end{figure}

\begin{figure}[t]
  \centering
  \includegraphics[width=0.9\textwidth,trim=2cm 2cm 8cm 2.8cm]{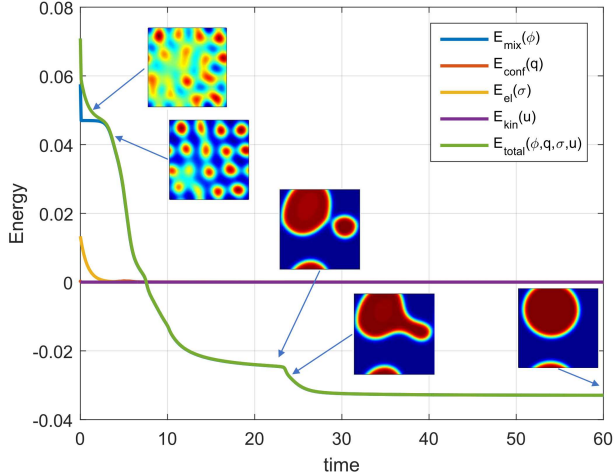}
  \caption{Energy evolution of the first numerical experiment
    corresponding to Figure~\ref{fig:ex1} with intermediate states of
    the phase separation.}
  \label{fig:en1}
\end{figure}

\begin{figure}[t]
  \centering
  \includegraphics[width=1\textwidth,trim=3.5cm 4cm 1.5cm 2cm]{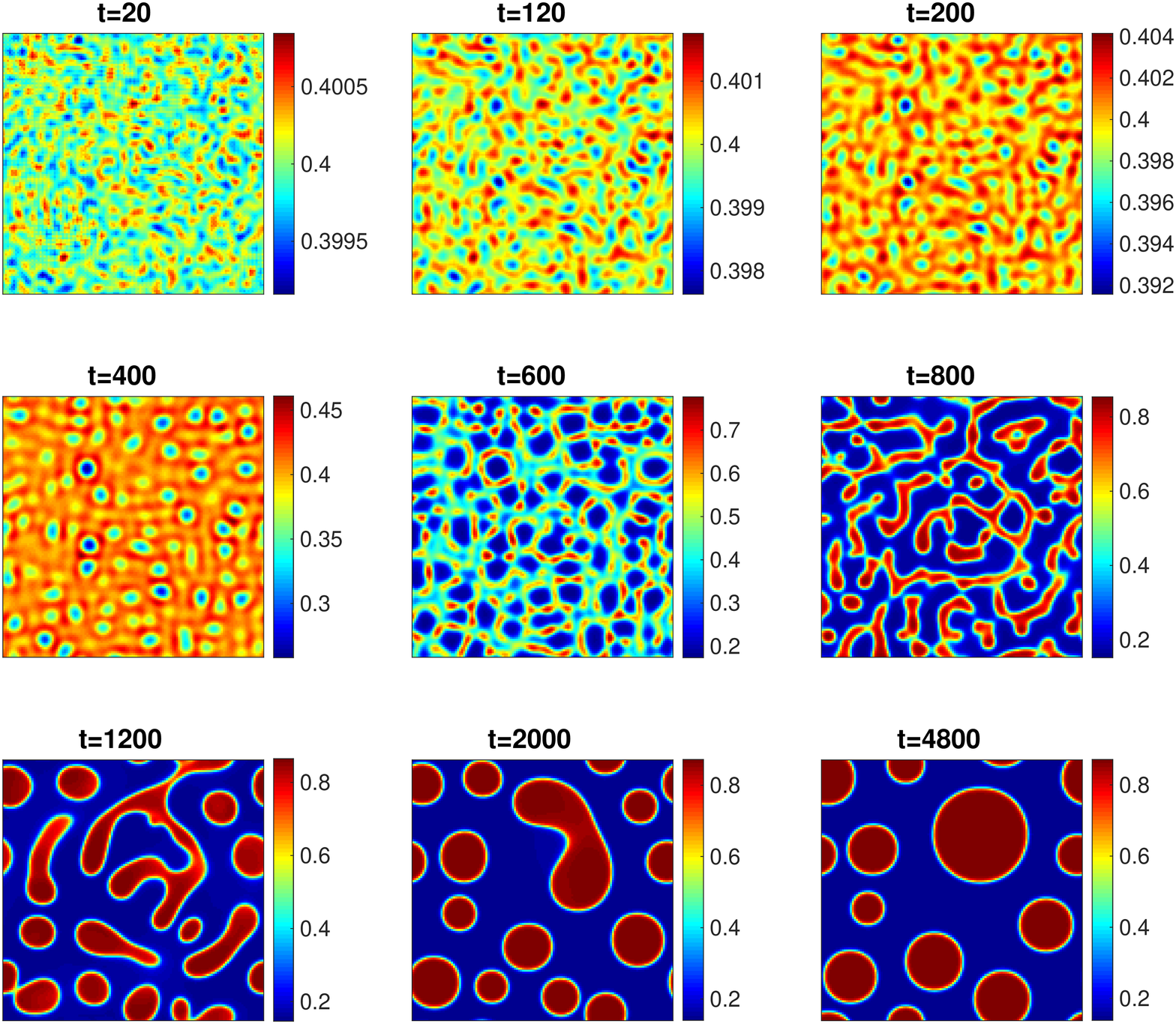}
  \caption{Time evolution of the polymer-solvent phase separation
    after a temperature quench with $\chi=2.\overline{54}, \phi_0 = 0.4$ and
    $\omega =[-0.001,0.001]$. The computational domain is
    $\Omega=[0,128]\times[0,128]$ and the interface width $\sqrt{C_0}
    = 1$. The time step is $\Delta t = 0.025$.}
  \label{fig:zze}
\end{figure}

\begin{figure}[t]
  \centering
  \includegraphics[width=1\textwidth,trim=1cm 0cm 1cm 0cm]{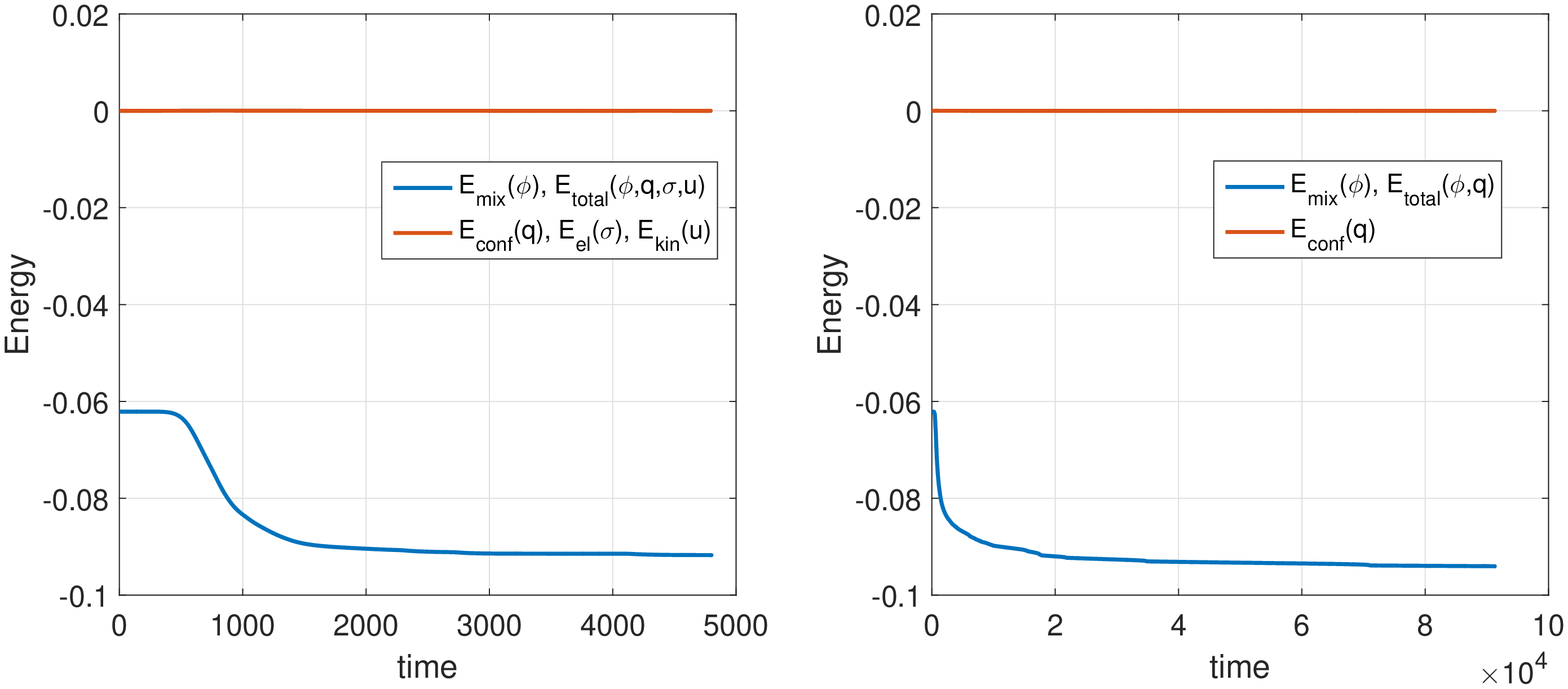}
  \caption{Energy evolution of the second numerical experiment
    corresponding to Figure~\ref{fig:zze} (left) and
    Figure~\ref{fig:simplified} (right).}
  \label{fig:en2}
\end{figure}

\begin{figure}[t]
  \centering
  \includegraphics[width=1\textwidth,trim=3.5cm 2.5cm 1.5cm 1cm]{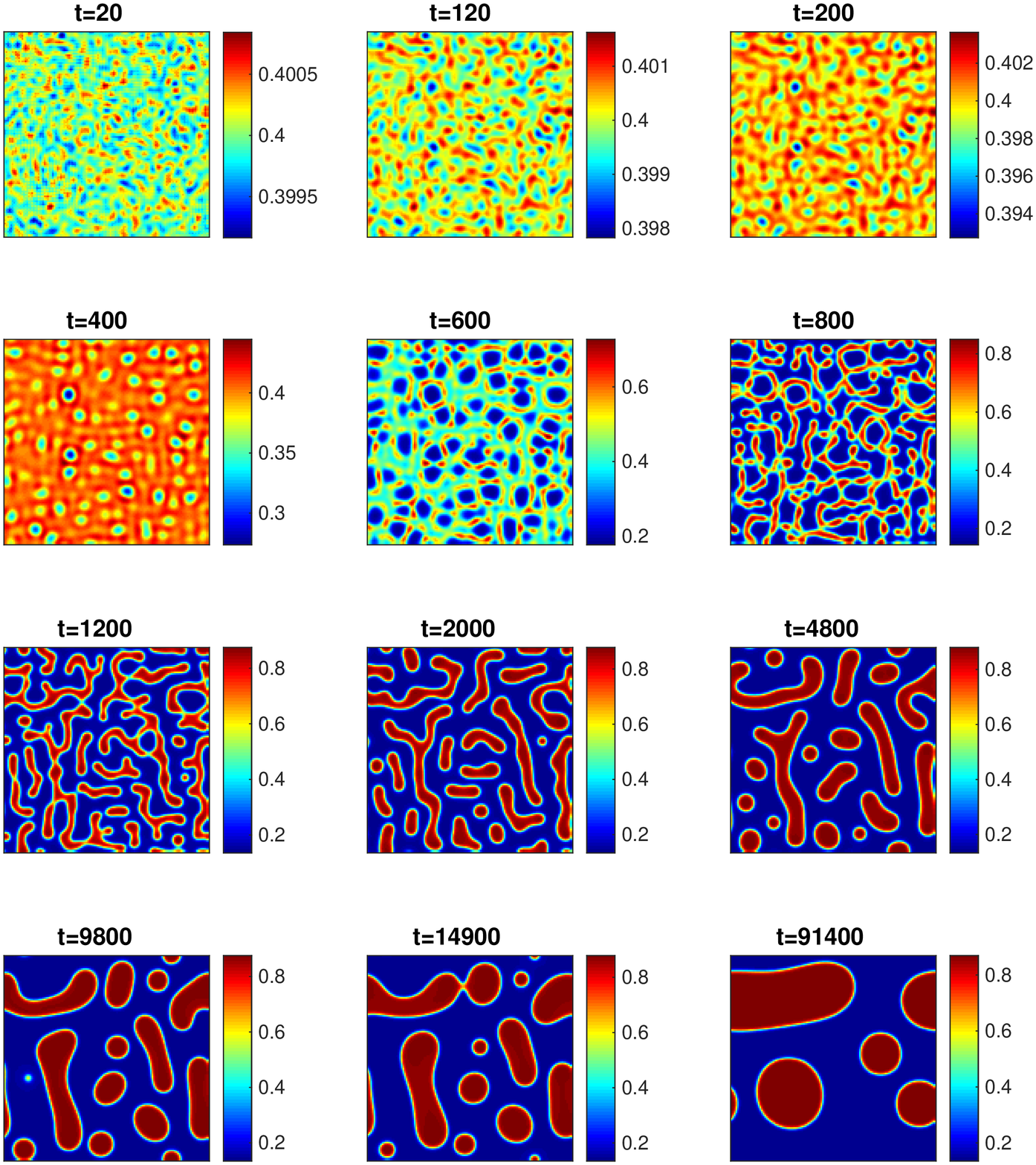}
  \caption{Time evolution of the polymer-solvent phase separation
    after a temperature quench, simulated by the simplified model
    utilizing the same parameter set used in the experiment from
    Figure~\ref{fig:zze}, except for the time step $\Delta t = 0.25$,
    since our second order scheme is used.}
  \label{fig:simplified}
\end{figure}

In the {\bf first numerical experiment} we solve numerically system
\eqref{eq:model} applying periodic boundary conditions.  The
computational domain $\Omega=[0,1]\times[0,1]$ is divided into
$128\times 128$ grid cells.  We follow the parameter set from
\textit{Gomez and Hughes}~\cite{gomez}.  The initial data of the
volume fraction $\phi(t=0)$ is taken to be a constant $\phi_0=0.4$
with a random perturbation distributed in $\omega = [-0.05,0.05]$ and
the initial velocities and bulk stress are set to zero.  The initial
value of the shear stress tensor is set to $\sig(t=0)= B_2(\phi(t=0))
(\sqrt{2} - 1 ) \one$, which implies the positivity definiteness of
the shear stress tensor. Further, we set the interface width $\sqrt{C_0} = \frac1{\sqrt{600}}$.
For the Flory-Huggins potential the degrees of polymerization are set to $n_p
= n_s = 1$ and the temperature-dependent Flory interaction $\chi=3$.  The bulk modulus is set to
$A_1(\phi)=M^0_b\left[1+\tanh\left(
    \frac{\cot(\pi\phi^*)-\cot(\pi\phi)}{\epsilon}\right)\right]+M^1_b$,
where $M^0_b=0.5$ and $M^1_b=1$, $\phi^*$ is set to be equal to the
initial average polymer volume fraction $\phi_0$ and $\epsilon=0.01$.
Furthermore, we set the mobility coefficient $M=10$ and the relaxation
coefficients to $\tau_b^0=10, \tau^0_s=5$ and $m^0_s=0.2$.

This experiment demonstrates phase separation by aggregation of the
polymer molecules towards droplets. Figure~\ref{fig:ex1} illustrates
time evolution of the volume fraction $\phi$.  The total energy is
(strictly) monotonically decreasing over time, which is related to the
surface minimization of the droplets and to droplets merging, see
Figure~\ref{fig:en1}.

The {\bf second experiment} has been proposed in \cite{zhou}. Here we
solve numerically both, the complete system \eqref{eq:model} as well
as the simplified model \eqref{eq:simp_model}. The computational
domain $\Omega=[0,128]\times[0,128]$ is divided to $128\times 128$
grid cells, initial volume fraction consists of $\phi_0 = 0.4$ and a
small random perturbation distributed in $\omega = [-0.001,0.001]$.
The interface thickness width $\sqrt{C_0} = 1$, which is already very
small having the size of one grid cell, and the Flory interaction $\chi=2.\overline{54}$.
The initial value of the shear stress is set to zero as in
\cite{zhou}. All other parameters are used as in the first experiment.

Figure~\ref{fig:zze} shows simulation of the complete system
\eqref{eq:model}, where the whole viscoelastic phase separation
process is exhibited.  In the earlier stage the polymer-rich phase
forms thin networklike structures.  The solvent-rich droplets grow and
coagulate. The area of the polymer-rich phase keeps decreasing. This
is the well-known volume-shrinking process in polymer phase
separation.  In the later stage polymer-rich networklike structures
are broken and the polymer-rich phase changes from being continuous to
being discontinuous. This process is called phase inversion,
cf.~\cite{zhou}.

Figure \ref{fig:simplified} illustrate the dynamics of the simplified
model \eqref{eq:simp_model}. We can clearly see that also this model
captures most important physical mechanism of the viscoelastic phase
separation. From time $t=600$ thin networklike structures formed by
the matrix-polymer-rich phase can be clearly recognized.

The presented experiments confirm the reliability of our newly
developed methods that preserve thermodynamic consistency of the
underlying physical model and dissipate free energy on the discrete
level, see Figures~\ref{fig:en1} and \ref{fig:en2}. Consequently, they
can be applied to model numerically complex polymeric mixtures and
provide a detailed view in the dynamics of a phase separation process
of a semi-dilute polymer-solvent mixture after a temperature quench,
including both key characteristics volume-shrinking and phase
inversion.

\section*{Conclusions}

In this paper we have derived and analysed new linear, energy
dissipative numerical schemes for viscoelastic phase separation. The
mathematical model is obtained through the variational principle as a
minimizer of the free energy. Consequently, the model consisting of
the Cahn-Hilliard equation describing the dynamics of interface
between polymer and Newtonian solvent and the Oldroyd-B for the
viscoelastic flow, can be understood as the gradient flow
corresponding to the total free energy.

The linearity of the numerical schemes increases the efficiency of
numerical simulations since there is no nonlinear iterative process
required.  The energy dissipative property is fundamental for
phase-separation problems and reflects their thermodynamic consistency
on the discrete level.  This property has been demonstrated
experimentally and proven theoretically up to the numerical
dissipation of the potential $ND_{phobic}$, which is small.

For the simplified model \eqref{eq:simp_model} the proposed numerical
scheme is second order. Numerical experiments confirm that the
simplified model can describe the most important physical properties
of the viscoelastic phase separation. The full system can be
approximated by the fully coupled scheme, Subsection~\ref{sec:coupled}
and the splitting scheme, Subsection~\ref{sec:splitting}. Both schemes 
yield analogous numerical solutions, but we opted here for the
splitting scheme, since it is more efficient computationally.

In future our aim is to develop hybrid schemes for multiscale models
of viscoelastic phase separation processes. Thus, our aim will be to
combine the proposed linear, energy dissipative schemes for
macroscopic models coupled with the combined Lattice-Boltzmann and
Molecular-Dynamics simulations of mesoscopic models for the
viscoelastic phase separation. We refer a reader to \cite{td} and the
references therein for more details on the latter scheme. We believe
that by such hybrid multiscale simulation the underlying physics will
become more clear and can provide deeper insight and perhaps also the
development of more refined and accurate macroscopic models.

\section*{Acknowledgements}

The present research has been supported by the German Science
Foundation (DFG) under the TRR-SFB 146 Multiscale Simulation Methods
for Soft Matter Systems. G. Tierra has been supported by
MTM2015-69875-P (Ministerio de Econom\'ia y Competitividad,
Spain). The authors gratefully acknowledge this support.


\bibliography{bibliography}

\end{document}